\documentclass[preprint,onecolumn,12pt]{IEEEtran}
\usepackage[a4paper,bindingoffset=0in,left=0.82in,right=0.82in,top=0.9in,bottom=0.9in,footskip=.3in]{geometry}
\usepackage{blindtext}

\usepackage{graphics}
\usepackage{graphicx}
\usepackage{epsfig}
\usepackage{amsmath} 
\usepackage{amssymb}
\usepackage{color}
\usepackage{amsfonts}
\usepackage{lineno} 
\usepackage{enumerate}

\allowdisplaybreaks

\newtheorem{theorem}{Theorem}
\newtheorem{definition}[theorem]{Definition}
\newtheorem{corollary}[theorem]{Corollary}
\newtheorem{lemma}[theorem]{Lemma}
\newtheorem{assumption}[theorem]{Assumption}
\newtheorem{proposition}[theorem]{Proposition}
\newtheorem{remark}[theorem]{Remark}
\newtheorem{claim}[theorem]{Claim}

\newcommand{\norm}[1]{\left\Vert #1\right\Vert}
\newcommand{\ssup}{{\sup}}
\newcommand{\abs}[1]{\left|#1\right|}

\def\field#1{\mathbb #1}%
\def\R{\field{R}}%
\def\Z{\mathbb{Z}}
\newcommand{\Rn}[1][n]{\R^{#1}}

\newcommand{\Rp}{\R_{\geq 0}}
\newcommand{\Rsp}{\R_{> 0}}
\newcommand{\Zp}{\Z_{\geq 0}}
\newcommand{\Zsp}{\Z_{> 0}}

\def\K{\mathcal{K}}%
\def\Kinf{\K_\infty}%
\def\KL{\mathcal{KL}}%
\def\KLL{\mathcal{KLL}}%
\def\dom{\mathrm{dom}}%
\def\C{\mathcal{C}}%
\def\D{\mathcal{D}}%
\def\H{\mathcal{H}}%
\def\A{\mathcal{A}}%
\def\X{\mathcal{X}}%
\def\U{\mathcal{U}}%
\DeclareMathOperator{\id}{id}

\DeclareMathOperator*{\esssup}{ess\,sup}

\ifCLASSINFOpdf
\else
\fi

\begin{document}

\title{A characterization of integral input-to-state stability for hybrid systems\\
\vspace{0.5cm}\large{Technical Report}}


\author{Navid Noroozi$^\textrm{a,b}\qquad$\thanks{$^\textrm{a}$Faculty of Computer Science and Mathematics, Passau University, Innstra$\ss$e 33, 94032 Passau, Germany}\thanks{$^\textrm{b}$Department of Electrical Engineering, University of Shahreza, 86149-56841 Shahreza, Iran} Alireza Khayatian$^\textrm{c}\qquad$\thanks{$^\textrm{c}$School of Electrical and Computer Engineering, Shiraz University, Shiraz, Iran} Roman Geiselhart$^\textrm{d}$\thanks{Institute of Measurement, Control and Microtechnology, University of Ulm, Albert-Einstein-Allee 41, 89081 Ulm, Germany}
\thanks{This work was mostly done while N. Noroozi was at Shiraz University.}}

\maketitle

\begin{abstract}
This paper addresses characterizations of integral input-to-state stability (iISS) for hybrid systems. In particular, we give a Lyapunov characterization of iISS unifying and generalizing the existing theory for pure continuous-time and pure discrete-time systems. Moreover, iISS is related to dissipativity and detectability notions. Robustness of iISS to sufficiently small perturbations is also
investigated. As an application of our results, we provide a maximum allowable sampling period guaranteeing iISS for sampled-data control systems with an emulated controller.
\end{abstract}

\begin{IEEEkeywords}
Integral input-to-state stability, hybrid systems, Lyapunov characterizations
\end{IEEEkeywords}

\section{Introduction}
There have been considerable attempts toward stability analysis of nonlinear systems in the presence of exogenous inputs over the last few decades. In particular, Sontag \cite{Sontag.1989} introduced the notion of input-to-state stability (ISS) which is indeed a generalization of $H_\infty$ stability for nonlinear systems. Many applications of ISS in analysis and design of feedback systems have been reported \cite{Sontag.2008}. A variant of ISS notion was introduced in \cite{Sontag.1998} extending $H_2$ stability to nonlinear systems. This generalization is called integral input-to-state stability (iISS) which was studied for continuous-time systems in \cite{Angeli.2000}, followed by an investigation into iISS of discrete-time systems in \cite{Angeli.1999}. As long as we are interested in stability analysis with respect to compact sets, it has been established that iISS is a more general concept rather than ISS and so every ISS system is also iISS while the converse is not necessarily true \cite{Sontag.1998}.

There is a wide variety of dynamical systems that can not be simply described either by differential or difference equations. This gives rise to so-called hybrid systems that combine both continuous-time (flows) and discrete-time (jumps) behaviors. Significant contributions concerned with modeling of hybrid systems have been developed in \cite{Goebel.2012}. In particular, a framework was developed in \cite{Goebel.2012} which not only models a wide range of hybrid systems, but also allows the study of stability and robustness of such systems.

This paper investigates iISS for hybrid systems modeled by the framework in \cite{Goebel.2012}.
Although the notion of iISS is well-understood for switched and impulsive systems (cf. \cite{Mancilla-Aguilar.2001} and \cite{Hespanha.2008} for more details), to the best of our knowledge, no further generalization of iISS being applicable to a wide variety of hybrid systems has been developed yet.
Toward this end, we provide a Lyapunov characterization of iISS unifying and generalizing the existing theory for pure continuous-time and pure discrete-time systems.
Furthermore, we relate iISS to dissipativity and detectability notions.
We also establish robustness of  the iISS property to vanishing perturbations.
We finally illustrate the effectiveness of our results by application to determination of a maximum allowable sampling period (MASP) guaranteeing iISS for sampled-data systems with an emulated controller.
To be more precise, we show that if a continuous-time controller renders a closed-loop system iISS, the iISS property of the closed-loop control system is preserved under an emulation-based digital implementation if the sampling period is taken less than the corresponding MASP.

The rest of this paper is organized as follows: First we introduce our notation in Section \ref{sec:Notation}. In Section \ref{sec:HybridSystemsandStabilityDefinitions}, a description of hybrid systems, solutions, and stability notions are given. The main results are presented in Section \ref{sec:MainResults}. Section~\ref{sec:examples} gives the iISS property of sampled-data control systems. Section \ref{sec:conclusions} provides the concluding remarks.

\section{Notation}\label{sec:Notation}
In this paper, $\Rp$ ($\Rsp$) and $\Zp$ ($\Zsp$) are nonnegative (positive) real and nonnegative (positive) integer numbers, respectively.
$\mathbb{B}$ is the open unit ball in $\Rn$.
The standard Euclidean norm is denoted by $\abs{\cdot}$.
Given a set $\mathcal{A}\subset\Rn$, $\overline{\mathcal{A}}$ denotes its closure. $\abs{x}_\mathcal{A}$ denotes $\inf\limits_{y \in \mathcal{A}}\abs{x-y}$ for a closed set $\mathcal{A}\subset\Rn$ and any point $x \in \Rn$.
Given an open set $\mathcal{X} \subset \Rn$ containing a compact set $\mathcal{A}$, a function $\omega \colon \mathcal{X} \to \Rp$ is a proper indicator for $\mathcal{A}$ on $\mathcal{X}$ if $\omega$ is continuous, $\omega(x)=0$ if and only if $x\in\A$, and $\omega(x_i) \to +\infty$ when either $x_i$ tends to the boundary of $\mathcal{X}$ or $\abs{x_i} \to +\infty$.
The identity function is denoted by $\id$.
Composition of functions from $\R$ to $\R$ is denoted by the symbol $\circ$. 

A function $\alpha \colon \Rp \to \Rp$ is said to be positive definite ($\alpha \in \mathcal{PD}$) if it is continuous, zero at zero and positive elsewhere.
A positive definite function $\alpha \colon \Rp \to \Rp$ is of class-$\K$ ($\alpha \in \K$) if it is strictly increasing.
It is of class-$\Kinf$ ($\alpha \in \Kinf$) if $\alpha \in \mathcal{K}$ and also $\alpha(s) \to +\infty$ if $s \to \infty$.
A continuous function $\gamma$ is of class-$\mathcal{L}$ ($\gamma \in \mathcal{L}$) if it is nonincreasing and $\lim_{s \to +\infty} \gamma (s) \to 0$.
A function $\beta \colon \Rp \times \Rp \to \Rp$ is of class-$\KL$ ($\beta \in \KL$), if for each  $s \geq 0$, $\beta(\cdot,s) \in \K$, and for each $r \geq 0$, $\beta (r,\cdot) \in \mathcal{L}$.
A function $\beta \colon \Rp \times \Rp \times \Rp \to \Rp$ is of class-$\KLL$ ($\beta \in \KLL$), if for each $s \geq 0$, $\beta (\cdot,s,\cdot) \in \mathcal{KL}$ and $\beta (\cdot,\cdot,s) \in \KL$.
The interested reader is referred to~\cite{Kellett.2014} for more details about comparison functions.

\section{Hybrid Systems and Stability Definitions} \label{sec:HybridSystemsandStabilityDefinitions}

Consider the following hybrid system with state $x \in \X$ and input $u \in \mathcal{U} \subset \R^d$ as follows
\begin{equation} \label{eq:e0}
\H := \left\{ \begin{array}{lcclr}
 \dot{x} &=& f(x,u) & \quad(x,u) \in \C \\
  x^+ &=& g(x,u) & \quad(x,u) \in \D
 \end{array} \right. .
\end{equation}
The flow and jump sets are designated by $\C$ and $\D$, respectively. We denote the system (\ref{eq:e0}) by a 6-tuple $\mathcal{H}=(f,g,\mathcal{C},\mathcal{D},\X,\mathcal{U})$.
Basic regularity conditions borrowed from \cite{Cai.2009} are imposed on the system $\mathcal{H}$ as follows
\begin{itemize}
    \item [A1)] $\X \subset \Rn$ is open, $\mathcal{U} \subset \R^d$ is closed and $\C$ and $\D$ are relatively closed sets in $\X \times \mathcal{U}$.
	 \item [A2)] $f \colon \C \to \Rn$ and $g \colon \D \to \X$ are continuous.
\item [A3)] For each $x \in \X$ and each $\epsilon \geq 0$, the set $\{ f(x,u) \mid u \in \mathcal{U} \cap \epsilon \overline{\mathbb{B}} \}$ is convex.
\end{itemize}
Here, we refer to the assumptions A1) to A3) as \emph{Standing Assumptions}. We note that the Standing Assumptions guarantee the \textit{well-posedness} of $\H$ (cf. \cite[Chapter 6]{Goebel.2012} for more details). Throughout the paper we suppose that the Standing Assumptions hold except otherwise stated.

The following definitions are needed in the sequel.
A subset $E \subset \Rp \times \Zp$ is called a compact hybrid time domain if $E=\bigcup_{j=0}^{J} ([t_j , t_{j+1}],j)$ for some finite sequence of real numbers $0 = t_0 \leq \cdots \leq t_{J+1}$.
We say $E$ is a hybrid time domain if, for each pair $(T,J) \in E$, the set $E \cap ([0,T] \times \{ 0,1, \dots ,J \})$ is a compact hybrid time domain.
For each hybrid time domain $E$, there is a natural ordering of points: given $(t,j),(t^\prime,j^\prime) \in E$, $(t,j) \preceq (t^\prime,j^\prime)$ if $t+j \leq t^\prime + j^\prime$, and $(t,j) \prec (t^\prime,j^\prime)$ if $t+j < t^{\prime}+j^{\prime}$. Given a hybrid time domain $E$, we define
\begin{align}
& \ssup_t E := \sup \{  t \in \Rp \colon \exists \, j \in \Zp \textrm{ such that } (t,j) \in E \} , & \nonumber \\
& \ssup_j E := \sup \{  j \in \Zp \colon \exists \, t \in \Rp \textrm{ such that } (t,j) \in E \} , & \nonumber \\
& \mathrm{length} (E) := \ssup_t E + \ssup_j E . & \nonumber
\end{align}
The operations $\ssup_t$ and $\ssup_j$ on a hybrid time domain $E$ return the supremum of the $\R$ and $\Z$ coordinates, respectively, of points in $E$.
A function defined on a hybrid time domain is called a hybrid signal.
Given a hybrid signal $x : \dom x \to \X$, for any $s \in \left[ 0, \ssup_t \dom x \right] \backslash \{+\infty\}$, $i(s)$ denotes the maximum index $i$ such that $(s,i) \in \dom x$, that is, $i(s) := \max \{ i \in \Zp \colon (s,i) \in \dom x \}$. A hybrid signal $x \colon \dom \,x \to \X$ is a hybrid arc if for each $j \in \Zp$, the function $t \mapsto z(t,j)$ is locally absolutely continuous on the interval $I^{j} := \{ t \colon (t,j) \in \dom x \}$.
A hybrid signal $u \colon \mathrm{dom} \, u \to \mathcal{U}$ is a hybrid input if for each $j \in \Zp$, $u(\cdot,j)$ is Lebesgue measurable and locally essentially bounded. 

Let a hybrid signal $v \colon \mathrm{dom} \, v \to \Rn$ be given. Let $(0,0), (t,j) \in \mathrm{dom} \, v$ such that $(0,0) \prec (t,j)$ and $\Gamma (v)$ denotes the set of $(t^{\prime},j^{\prime}) \in \mathrm{dom} \, v$ so that $(t^{\prime},j^{\prime}+1) \in \mathrm{dom} \, v$. Define
\begin{align}
\norm{v_{(t,j)}}_\infty \!\!:=\! \max \!\Big\{\! & \esssup_{\scriptsize{\begin{array}{c}(t^\prime,j^\prime) \in \dom v \backslash \Gamma (v),\\ (0,0) \preceq (t^\prime,j^\prime) \preceq (t,j) \end{array}}} \!\!\!\!\!\!\!\!\!\!\!\!\abs{v(t^\prime,j^\prime)} ,  \sup_{\scriptsize{\begin{array}{c}(t^\prime,j^\prime) \in \Gamma (v), \\(0,0) \preceq (t^\prime,j^\prime) \preceq (t,j)\end{array}}} \!\!\!\!\!\!\!\!\!\!\!\! \abs{v(t^\prime,j^\prime)} \Big\} . & \nonumber
\end{align}

Let $\gamma_1,\gamma_2 \in \K$ and let $u \colon \dom\, u \to \mathcal{U}$ be a hybrid input such that for all $(t,j) \in \dom \,u$ the following holds
\begin{align}
& \norm{u_{(t,j)}}_{\gamma_1,\gamma_2} := \int_0^t \gamma_1 (\abs{u(s,i(s))}) \mathrm{d} s + \!\!\!\!\!\!\!\!\!\!\!\! \sum_{\scriptsize{\begin{array}{c}(t^\prime,j^\prime) \in \Gamma(u),\\(0,0) \preceq (t^{\prime},j^{\prime}) \prec (t,j)\end{array}}} \!\!\!\!\!\!\!\!\!\!\!\!\gamma_2 (\abs{u(t^\prime,j^\prime)})  < + \infty . & \nonumber
\end{align}
We denote the set of all such hybrid inputs by $\mathcal{L}_{\gamma_1,\gamma_2}$. Also, if
$\norm{u_{(t,j)}}_{\gamma_1,\gamma_2} < r$ for some $r > 0$ and all $(t,j) \in \dom \,u$, we write $u \in \mathcal{L}_{\gamma_1,\gamma_2} (r)$. Assume that the hybrid input $u \colon \dom \,u \to \mathcal{U}$. For each $T \in \left[ 0, \mathrm{length} (\dom \,u) \right] \backslash \{+\infty\}$, the hybrid input $u_T \colon \dom \,u \to \mathcal{U}$ is defined by
\begin{eqnarray}
u_T (t,j) = \left\{ \begin{array}{lr}
 u(t,j) & \qquad t+j \leq T \\
 0 & \qquad t+j > T 
 \end{array} \right. \nonumber
\end{eqnarray}
and is called the $T$-truncation of $u$. The set $\mathcal{L}_{\gamma_1,\gamma_2}^e$ $\left(\mathcal{L}_{\gamma_1,\gamma_2}^e (r)\right)$ consists of all hybrid inputs $u(\cdot,\cdot)$ with the property that for all $T \in [0,\infty)$, $u_T \in \mathcal{L}_{\gamma_1,\gamma_2} \left(u_T \in \mathcal{L}_{\gamma_1,\gamma_2} (r) \right)$, and is called the extended $\mathcal{L}_{\gamma_1,\gamma_2}$-space.

A hybrid arc $x \colon \dom\,x \to \X$ and a hybrid input $u \colon \dom\,u \to \mathcal{U}$ is a solution pair $(x,u)$ to $\H$ if $\dom x = \dom u$, $(x(0,0),u(0,0)) \in \C \cup \D$, and
\begin{itemize}
\item for each $j \in \Zp$, $(x(t,j),u(t,j)) \in \C$ and $\dot{x}=f(x(t,j),u(t,j))$ for almost all $t \in I^j$ where $I^{j}$ has nonempty interior;
\item for all $(t,j) \in \Gamma(x)$, $(x(t,j),u(t,j)) \in \D$ and $x(t,j+1)=g(x(t,j),u(t,j))$.
\end{itemize}
A solution pair $(x,u)$ to $\H$ is maximal if it cannot be extended, it is complete if $\dom\,x$ is unbounded. A maximal solution to $\mathcal{H}$ with the initial condition $\xi := x(0,0)$ and the input $u$ is denoted by $x(\cdot,\cdot,\xi,u)$. The set of all maximal solution pairs $(x,u)$ to $\H$ with $\xi := x(0,0) \in \X$ is designated by $\varrho^u (\xi)$.

\subsection{Stability Notions}

Given the system $\H$ and a nonempty and compact $\A \subset \X$, then $\A$ is called
\begin{itemize}
\item 0-input pre-stable if for any $\epsilon > 0$ there exists $\delta > 0$ such that each solution pair $(x,0) \in \varrho^u (\xi)$ with $\abs{\xi}_\A \leq \delta$ satisfies $\abs{x(t,j,\xi,0)}_\A \leq \epsilon$ for all $(t,j) \in \dom\,x$.
\item 0-input pre-attractive if there exists $\delta > 0$ such that each solution pair $(x,0) \in \varrho^u (\xi)$ with $\abs{\xi}_\A \leq \delta$ is bounded (with respect to $\X$) and if it is complete then $\lim_{(t,j) \in \dom \, x , \, t+j \to +\infty}$ $\abs{x(t,j,\xi,0)}_\A \to 0$. 
\item 0-input pre-asymptotically stable (pre-AS) if it is both 0-input pre-stable and 0-input pre-attractive.
\item 0-input asymptotically stable (AS) if it is 0-input pre-AS and there exists $\delta > 0$ such that each solution pair $(x,0) \in \varrho^u (\xi)$ with $\abs{\xi}_\A \leq \delta$ is complete.
\end{itemize}
It should be noted that the prefix "pre-" emphasizes that not every solution requires to be complete. If all solutions are complete, then we drop the pre.
\begin{definition} \label{D:1}
Let $\A \subset \X$ be a compact set.
Also, let $\omega$ be a proper indicator for~$\A$ on~$\X$.
The hybrid system $\H$ is said to be pre-integral input-to-state stable $($pre-iISS$)$ with respect to $\A$ if there exist $\alpha \in \Kinf$, $\gamma_1,\gamma_2 \in \K$ and $\tilde{\beta} \in \KLL$ such that for all $u \in \mathcal{L}_{\gamma_1,\gamma_2}^e$, all $\xi \in \X$, and all $(t,j) \in \dom \, x$, each solution pair $(x,u)$ to $\H$ satisfies
\begin{align} \label{eq:e1}
& \alpha(\omega(x(t,j,\xi,u))) \leq \tilde{\beta} (\omega(\xi),t,j) + \norm{u_{(t,j)}}_{\gamma_1,\gamma_2} . &
\end{align}
\hfill$\Box$
\end{definition}
\begin{remark}
We point out that $\alpha$ on the left-hand side of \eqref{eq:e1} is redundant. In particular, $\H$ is pre-iISS with respect to $\A$ if and only if there exist $\eta,\gamma_1,\gamma_2 \in \K$ and $\beta \in \KLL$ satisfying
\begin{align*}
& \omega(x(t,j,\xi,u)) \leq \beta (\omega(\xi),t,j) + \eta\left(\norm{u_{(t,j)}}_{\gamma_1,\gamma_2}\right) . &
\end{align*}
We, however, place emphasis on \eqref{eq:e1} for two reasons: firstly, \eqref{eq:e1} is consistent with the continuous-time and discrete-time counterparts in \cite{Angeli.2000,Angeli.1999}. Secondly, \eqref{eq:e1} simplifies exposition of proofs.
\hfill$\Box$
\end{remark}
\begin{definition} \label{D:2}
Given a compact set $\A \subset \X$, let $\omega$ be a proper indicator for $\A$ on $\X$. A smooth function $V \colon \X \to \Rp$ is called an iISS-Lyapunov function with respect to $(\omega,\abs{\cdot})$ for (\ref{eq:e0}) if there exist functions $\alpha_1,\alpha_2 \in \Kinf$, $\sigma \in \K$, and $\alpha_3 \in \mathcal{PD}$ such that
\begin{align}
\alpha_1 (\omega(\xi)) & \leq V(\xi) \leq \alpha_2 (\omega(\xi)) \quad\,\qquad  \forall \xi \in \X, \label{eq:e3} \\
\langle \nabla V (\xi),f(\xi,u) \rangle & \leq -\alpha_{3}(\omega(\xi)) + \sigma(\abs{u}) \qquad \forall (\xi,u) \in \C, \label{eq:e4} \\
V (g(\xi,u)) - V (\xi) & \leq - \alpha_{3}(\omega(\xi)) + \sigma(\abs{u}) \qquad \forall (\xi,u) \in \D . \label{eq:e5}
\end{align}
\hfill$\Box$
\end{definition}
\begin{definition} \label{D:3} (\cite{Angeli.2000})
A positive definite function $W \colon \X \to \Rp$ is called a semi-proper if there exist $\pi \in \K$, and a proper positive definite function $W_0$ such that $W (\cdot) = \pi ( W_0 (\cdot) )$.
\hfill$\Box$
\end{definition}

The following definitions are required to relate pre-iISS to the hybrid invariance principle \cite{Goebel.2012}. 

\begin{definition} \label{D:5} (\cite[Definition 6.2]{Sanfelice.2007})
Given sets $\A , K \subset \X$, the distance to $\A$ is 0-input detectable relative to $K$ for $\H$ if every complete solution pair $(x,0)$ to $\H$ such that $x(t,j) \in K$ for all $(t,j) \in \dom x$ implies that $\lim_{(t,j)\to+\infty,(t,j)\in \dom x} \omega(x(t,j)) = 0$ where $\omega$ is a proper indicator for $\A$ on $\X$.
\hfill$\Box$
\end{definition}

\begin{definition} \label{D:6}
Let $\omega$ be a proper indicator for $\A$ on $\X$. $\H$ is said to be smoothly dissipative with respect to $\A$ if there exists a smooth function $V \colon \X \to \Rp$, called a storage function, functions $\alpha_4,\alpha_5 \in \Kinf$, $\sigma \in \K$, and a continuous function $\rho : \X \to \Rp$ with $\rho(\xi) = 0$ for all $\xi \in \A$ such that
\begin{eqnarray}
\alpha_4 (\omega(\xi)) & \leq & V(\xi) \leq \alpha_5 (\omega(\xi)) \quad \forall \xi \in \X , \label{eq:e03} \\
\langle \nabla V (\xi),f(\xi,u) \rangle & \leq & - \rho (\xi) + \sigma(\abs{u}) \quad \forall (\xi,u) \in \C , \label{eq:e04} \\
V(g(\xi,u)) - V(\xi) & \leq & - \rho (\xi) + \sigma(\abs{u}) \quad \forall (\xi,u) \in \D . \label{eq:e05}
\end{eqnarray}
\hfill$\Box$
\end{definition}

We note that Definition~\ref{D:6} subsumes Definition~\ref{D:2} as a special case. As we will see later (cf. Theorem~\ref{T:1}), the existence of a storage function $V$ plus the 0-input detectability relative to $K$ is equivalent to the existence of an iISS-Lyapunov function.

\section{Main Results}\label{sec:MainResults}

This section addresses equivalences for pre-iISS. Particularly, a Lyapunov characterization of pre-iISS together with other related notions is presented.
Before proceeding further, we recall~\cite[Lemma IV.1]{Angeli.2000} on positive definite functions, which is used later.
\begin{lemma}\label{L:1}
Given $\rho \in \mathcal{PD}$, there exist $\rho_1 \in \Kinf$ and $\rho_2 \in \mathcal{L}$ such that $\rho (r) \geq \rho_1 (r) \rho_2 (r)$, $\forall r \geq 0$.
\hfill$\Box$
\end{lemma}

The lemma below is a generalization of \cite[Lemma IV.2]{Angeli.2000} for hybrid systems, that is used to give the proof of our main result (see Section~\ref{S:1}).  
\begin{lemma} \label{L:2}
Let $\rho \in \mathcal{PD}$ with $\rho(r) < r$ for all $r > 0$, and $z \colon \dom z \to \R$ be a hybrid arc with $z(0,0) \geq 0$. Consider a hybrid signal $v \colon \dom v \to \Rp$ such that $\dom v = \dom z$ and for each $j$, $v(\cdot,j)$ is continuous. Furthermore, assume that
\begin{itemize}
\item for almost all $t$ such that $(t,j) \in \dom z \backslash \Gamma(z)$
\begin{equation} \label{eq:e6}
\dot{z}(t,j) \leq -\rho (\mathrm{max} \{ z(t,j) + v(t,j) , 0 \})
\end{equation}
\item for all $(t,j) \in \Gamma(z)$ it holds that
\begin{equation} \label{eq:e7} 
z(t,j+1) - z(t,j) \leq -\rho (\mathrm{max} \{ z(t,j) + v(t,j) , 0 \}) .
\end{equation}
\end{itemize}
Then, there exists $\beta \in \mathcal{KLL}$ such that 
\begin{equation} \label{eq:e8}
z(t,j) \leq \mathrm{max} \{ \beta ( z (0,0), t , j ),\norm{v_{(t,j)}}_\infty \}  \qquad \forall (t,j) \in \dom \, z .
\end{equation}
\end{lemma}
\begin{proof} See Appendix~\ref{Ap:E}. \end{proof}
Given a set $S \subset \X \times \U$, we denote $\Pi_0 (S) := \{ x \in \X : (x,0) \in S \}$. Here is the main result of this paper. 
\begin{theorem} \label{T:1}
Let $\A \subset \X$ be a compact set.
Also, let $\omega$ be a proper indicator for $\A$ on $\X$.
Suppose that the Standing Assumptions hold.
Also, assume that $\Pi_0 (\C) \cup \Pi_0 (\D) = \X$. Then the following are equivalent
\begin{enumerate}[{(}i{)}]
    \item\label{item:iiss} $\mathcal{H}$ is pre-iISS with respect to $\A$.
    \item\label{item:lyapunov-iiss} $\mathcal{H}$ admits a smooth iISS-Lyapunov function with respect to $(\omega,\abs{\cdot})$.
    \item\label{item:0-detect} $\mathcal{H}$ is smoothly dissipative with respect to $\A$ and the distance to $\A$ is 0-input detectable relative to $\{ \xi \in \X \colon \rho (\xi) = 0 \}$ with $\rho$ as in \eqref{eq:e04} and \eqref{eq:e05}.

    \item\label{item:0-gas} $\mathcal{H}$ is 0-input pre-AS and $\mathcal{H}$ is smoothly dissipative with respect to $\A$ with $\rho \equiv 0$.
\end{enumerate}
\end{theorem}
\begin{proof}
We show that $(\ref{item:lyapunov-iiss}) \Rightarrow (\ref{item:iiss})$ in Subsection~\ref{S:1}.
We also give a proof of the implication $(\ref{item:iiss}) \Rightarrow (\ref{item:lyapunov-iiss})$ in Subsection~\ref{S:2}.
The implication $(\ref{item:0-gas}) \Rightarrow (\ref{item:lyapunov-iiss})$ immediately follows from the combination of Proposition~\ref{P:2}
(see below) and Definition~\ref{D:6}.
To see the implication $(\ref{item:lyapunov-iiss}) \Rightarrow (\ref{item:0-detect})$, let the iISS Lyapunov function $V$ be a storage function with $\rho(x) := \alpha_3 (\omega(x))$ and $\alpha_3$ as in (\ref{eq:e4}) and (\ref{eq:e5}).
So $\mathcal{H}$ is smoothly dissipative. Moreover, the distance to $\A$ is 0-input detectable relative to $\{ \xi \in {\X} \colon \rho (x) = 0 \}$ because $\rho(x) = 0$ implies that $x \in \mathcal{A}$. Finally the implication $(\ref{item:0-detect}) \Rightarrow (\ref{item:0-gas})$ is provided as follows: Let $V$ be a storage function.
Also, assume that $u \equiv 0$. According to~\cite[Theorem 23]{Goebel.2009}, $\A$ is 0-input pre-stable.
To show 0-input pre-attractivity of $\A$, consider a complete solution pair $(x,0)$ to $\H$, that is bounded by 0-input pre-stability of $\A$.
We first note that $\H$ satisfying the Standing Assumptions and $u \equiv 0$ imply that the invariance principle for hybrid systems (e.g. Corollary 8.4 in \cite{Goebel.2012}) can be applied.
According to \cite[Corollary 8.4]{Goebel.2012}, there exists some $r \geq 0$ such that every complete solution $(x,0)$ to $\H$ converges to the largest weakly invariant set contained in
\begin{align} \label{eq:det-set}
& \big\{ \xi : V(\xi) = r \big\} \cap \big( \rho_\C^{-1} (0) \cup \rho_\D^{-1} (0)\big) &
\end{align}
where $\rho_\C^{-1} (0) := \{ \xi \in \C : \rho (\xi) = 0 \}$ and $\rho_\D^{-1} (0) := \{ \xi \in \D : \rho (\xi) = 0 \}$.
It follows from the 0-input detectability relative to $\{ \xi \in \X \colon \rho (\xi) = 0 \}$ that every complete solution contained in the set (\ref{eq:det-set}) converges to $\A$.
Moreover, from \eqref{eq:e03}, the only invariant set in \eqref{eq:det-set} is obtained for $r=0$.
As the set \eqref{eq:det-set} lies in $\A$ for $r=0$, then $\A$ is 0-input pre-attractive.
Eventually, we note that smooth dissipativity of $\mathcal{H}$ with respect to $(\omega,\abs{\cdot})$ with $\rho \equiv 0$ is obviously satisfied. This completes the proof.
\end{proof}
\begin{remark}
The assumption $\Pi_0 (\C) \cup \Pi_0 (\D) = \X$ means that the union of the flow set and the jump set generated by the disturbance-free system covers $\X$.
As shown in~\cite[Section IV]{Cai.2008}, there are hybrid systems not satisfying the assumption, hybrid systems with logic variables for instance.
This assumption could be relaxed at the expense of further technicalities following similar lines as in the proof of \cite[Theorem 7.31]{Goebel.2012}.
However, we do not focus on that as it makes the proofs much more complicated without considerable appreciation.
\hfill$\Box$
\end{remark}
\subsection{Illustrative example}
Here we verify iISS of a hybrid system using an iISS Lyapunov function.
Consider a first-order integrator
\begin{align}
\dot x_p = u ,
\end{align}
where $u\in \R$ is the control input to the system.
We aim to control the system using a reset controller under input constraints (i.e. $|u| \leq \overline u$ for some given $\overline u >0$).
As shown in~\cite{Nesic.2008b}, designing a reset controller subjected to disturbances and input constraints leads to a hybrid system of the form~\eqref{eq:e0} as follows
\begin{subequations} \label{eq:res-sys}
\begin{align}
& \left.
\begin{array}{rcl}
\dot x_p & = & \lambda_p \arctan(x_p) + b \arctan(x_c) + w \\
\dot x_c & = & \lambda_c \arctan(x_c) + k \arctan(x_p)
\end{array}
\right\}  (x,w) \in \C , \\
& \left.
\begin{array}{rcl}
x_p^+ & = & x_p \\
x_c^+ & = & 0
\end{array}
\right\}  (x,w) \in \D ,
\end{align}
\end{subequations}
where $x := (x_p,x_c)$ is the sate of the closed-loop system, $w \in \R$ is the disturbance input, $\C = \{ (x,w) \in \R^2 \times \R : x_p ( x_c - x_p) \leq 0 \}$, $\D = \{ (x,w) \in \R^2 \times \R : x_p ( x_c -  x_p) \geq 0 \}$, and the constants $b , k > 0$ and $\lambda_p , \lambda_c < 0$ are chosen later.
From $\D$, the output of controller is reset to zero whenever $x_p ( x_c - x_p) \geq 0$.
Note that for sufficiently large $w$ each solution to the system is unbounded, which shows that the system is \emph{not} ISS.
\begin{corollary}
Consider system~\eqref{eq:res-sys}. Given $b , k > 0$ and $\lambda_p , \lambda_c < 0$, assume that there exist real positive numbers $c_1,c_2>0$ such that
\begin{align}\label{eq:stability-conditions}
c_1 \lambda_p + b c_1 + k c_2 \leq 0 , \;\; c_2 \lambda_c + k c_2 + b c_1 \leq 0 .
\end{align}
Take the proper indicator $\omega (\cdot) = \abs{\cdot}$. Then system~\eqref{eq:res-sys} is pre-iISS with respect to the origin.
\end{corollary}
\begin{proof}
Take the following iISS Lyapunov function candidate
$$
V(x) = c_1 x_p \arctan(x_p) + c_2 x_c \arctan(x_c) .
$$
Obviously, $V$ satisfies~\eqref{eq:e03} for some appropriate $\alpha_1,\alpha_2 \in \Kinf$ and $\omega (\cdot) = \abs{\cdot}$.
Picking $(x,w) \in \C$, we have
\begin{align*}
\langle \nabla V , f(x,w) \rangle = & c_1 \Big[ \arctan (x_p) \big( \lambda_p \arctan (x_p) + b \arctan (x_c) + w \big) \\
& \;\;\;\;\; + \frac{x_p}{1+x_p^2} \big( \lambda_p \arctan (x_p) + b \arctan (x_c) + w \big)
\Big] \\
& + c_2 \Big[ \arctan (x_c) \big( \lambda_c \arctan (x_c) + k \arctan (x_p) \big) \nonumber\\
&\qquad\;\;+ \frac{x_c}{1+x_c^2} \big( \lambda_c \arctan (x_c) + k \arctan (x_p) \big) \Big] .
\end{align*}
Using Young's inequality and the facts that $\abs{\arctan(s)} \leq \pi/2$ and $\abs{s}/(1+s^2) \leq 1$ for all $s \in \R$ give
\begin{align*}
\langle \nabla V , f(x,w) \rangle \leq  & \big( c_1 \lambda_p + 0.5 b c_1 + k c_2 \big) [\arctan (x_p)]^2 + c_1 \lambda_p \frac{x_p \arctan(x_p)}{1+x_p^2}\\
& 
+ \frac{c_1 b}{2} \frac{x_p^2}{1+x_p^2} + \big(c_2 \lambda_c + 0.5 k c_2 + b c_1 \big) [\arctan (x_c)]^2 \\
& + c_2 \lambda_c \frac{x_c \arctan(x_c)}{1+x_c^2}
+ \frac{c_2 k}{2} \frac{x_c^2}{1+x_c^2} + \frac{c_1 (\pi+1)}{2} \abs{w} .
\end{align*}
From the fact that $\frac{s^2}{1+s^2} \leq [\arctan(s)]^2$ for all $s \in \R$, we have
\begin{align}
\langle \nabla V , f(x,w) \rangle \leq & \big( c_1 \lambda_p + b c_1 + k c_2 \big) [\arctan (x_p)]^2 
+ c_1 \lambda_p \frac{x_p \arctan(x_p)}{1+x_p^2} \nonumber\\
& \!+\! \big(c_2 \lambda_c + k c_2 + b c_1 \big) [\arctan (x_c)]^2
\! + \! c_2 \lambda_c \frac{x_c \arctan(x_c)}{1+x_c^2} \! + \! \frac{c_1 (\pi+1)}{2} \abs{w} \nonumber\\
\leq & \big( c_1 \lambda_p + b c_1 + k c_2 \big) [\arctan (x_p)]^2 
+ \big(c_2 \lambda_c + k c_2 + b c_1 \big) [\arctan (x_c)]^2 \nonumber\\
& +\frac{c_1 (\pi+1)}{2} \abs{w} .
 \label{eq:V-flows}
\end{align}
Now we consider jump equations on the set $\D$. For any $(x,w) \in \D$ we get
\begin{align*}
V(g(x)) - V(x) & = - c_2 x_c \arctan(x_c) \\
& = - \rho x_p \arctan(x_p) - c_2 x_c \arctan(x_c) + \rho x_p \arctan(x_p) , 
\end{align*}
where $0< \rho < c_2$. Note that $(x,w) \in \D$ implies that $x_p \arctan(x_p) \leq x_c \arctan(x_c)$. So we have
\begin{align}
V(g(x)) - V(x) & \leq - \rho x_p \arctan(x_p) - c_2 x_c \arctan(x_c) + \rho x_c \arctan(x_c)  \nonumber\\
& = - \rho x_p \arctan(x_p) - (c_2-\rho) x_c \arctan(x_c) . \label{eq:V-jumps}
\end{align}
It follows from~\eqref{eq:stability-conditions}, \eqref{eq:V-flows} and \eqref{eq:V-jumps} that $V$ is an iISS Lyapunov function for system~\eqref{eq:res-sys}.
\end{proof}
Finding an iISS Lyapunov function is not always easy.
Alternatively, either item~(\ref{item:0-detect}) or~(\ref{item:0-gas}) can be used to conclude the iISS property; see Section~\ref{sec:examples}.
\subsection{Proof of the implication $(\ref{item:lyapunov-iiss}) \Rightarrow (\ref{item:iiss})$} \label{S:1}

Consider a solution pair $(x,u)$ to $\mathcal{H}$. Given~\eqref{eq:e4} and~\eqref{eq:e5}, we have
\begin{equation*}
\langle \nabla V (x(t,j)),f(x(t,j),u(t,j)) \rangle \leq - \alpha_3 (\omega(x(t,j))) + \sigma(\abs{u(t,j)})
\end{equation*}
for almost all $t$ such that $(t,j) \in \dom\,x \backslash \Gamma(x)$; and
\begin{equation*}
V\left(g(x(t,j),u(t,j))\right) - V(x(t,j)) \leq - \alpha_3 (\omega(x(t,j))) + \sigma(\abs{u(t,j)}) 
\end{equation*}
for all $(t,j) \in \Gamma(x)$. Applying Lemma~\ref{L:1} to $\alpha_3$, there exist $\rho_1 \in \Kinf$ and $\rho_2 \in \mathcal{L}$ such that
\begin{equation*}
\langle \nabla V (x(t,j)),f(x(t,j),u(t,j)) \rangle \leq - {\rho _1}\left( {\omega \left( x(t,j)  \right)} \right)\,{\rho _2}\left( {\omega \left( x(t,j)  \right)} \right) + \sigma(\abs{u(t,j)}) 
\end{equation*}
for almost all $t$ such that $(t,j) \in \dom\,x \backslash \Gamma(x)$; and
\begin{equation*}
V(g(x(t,j),u(t,j))) - V(x(t,j)) \leq - {\rho _1}\left( {\omega \left( x(t,j)  \right)} \right)\,{\rho _2}\left( {\omega \left( x(t,j)  \right)} \right) + \sigma(\abs{u(t,j)}) 
\end{equation*}
for all $(t,j) \in \Gamma(x)$. Exploiting~\eqref{eq:e3} and letting $\tilde{\rho} ( \cdot ) := \rho_1 \circ \alpha_2^{- 1} ( \cdot ) \rho_2 \circ\alpha_1^{- 1} ( \cdot )$ yield
\begin{equation}
\langle \nabla V (x(t,j)),f(x(t,j),u(t,j)) \rangle \leq - \tilde{\rho} ( V ( \xi ) ) + \sigma(\abs{u}) \label{eq:e14}
\end{equation}
for almost all $t$ such that $(t,j) \in \dom\,x \backslash \Gamma(x)$; and
\begin{equation}
V(g(x(t,j),u(t,j))) - V(x(t,j)) \leq - \tilde{\rho} ( V ( \xi ) ) + \sigma(\abs{u}) \label{eq:e15}
\end{equation}
for all $(t,j) \in \Gamma(x)$. Define the hybrid arcs $z$ and $v$ by
\begin{align}
& z(t,j)  := V(x(t,j)) - v(t,j) , \label{eq:e71} \\
& v(t,j) := \int_{0}^{t} \sigma ( \abs{u(s,i(s))} ) \, \mathrm{d} s + \!\!\!\!\!\!\!\!\!\!\!\! \sum_{\scriptsize{\begin{array}{c}(t^\prime,j^\prime) \in \Gamma(u),\\(0,0) \preceq (t^{\prime},j^{\prime}) \prec (t,j)\end{array}}} \!\!\!\!\!\!\!\!\!\!\!\! \sigma (\abs{u(t^\prime,j^\prime)}) . \label{eq:e72}
\end{align}
It should be pointed out that the hybrid arcs $z$ and $v$ are defined on the same hybrid time domain $\mathrm{dom}\,x$ because, by the assumption, $\mathrm{dom}\,x = \mathrm{dom}\,u$. It follows from (\ref{eq:e14}), ~(\ref{eq:e71}) and~(\ref{eq:e72}) that the following holds for almost all $t$ such that $(t,j) \in \mathrm{dom} \, z \backslash \Gamma(z)$
\begin{align} \label{eq:e73}
\dot{z} (t,j) & \leq - \tilde{\rho} (V(x(t,j))) = - \tilde{\rho} ( \max \{ z(t,j) + v(t,j) , 0 \} ) .
\end{align}
From (\ref{eq:e15}),~(\ref{eq:e71}) and~(\ref{eq:e72}), we have for all $(t,j) \in \Gamma(z)$
\begin{align}
& z (t,j + 1) - z (t,j) \leq - \tilde{\rho} ( \max \{ z(t,j) + v(t,j) , 0 \} ) . & \label{eq:e74}
\end{align}
It follows from (\ref{eq:e73}), (\ref{eq:e74}) and Lemma~\ref{L:2}, there exists $\beta \in \mathcal{KLL}$ such that
\begin{align} \label{eq:e75}
z(t,j) \leq & \max \{ \beta (z(0,0),t,j), \norm{v_{(t,j)}}_\infty \} \leq \beta (z(0,0),t,j) + \norm{v_{(t,j)}}_\infty  &
\end{align}
for all $(t,j) \in \dom \, z$. An immediate consequence from~(\ref{eq:e71}),~(\ref{eq:e72}) and the facts that $z(0,0) = V(x(0,0))$ and $\norm{v_{(t,j)}}_\infty = v(t,j)$ is
\begin{align}
V ( x (t,j) ) \leq & \beta \left( {V\left( {{x(0,0)}} \right),t,j} \right) + 2 \int_{0}^{t} {\sigma \left( {\abs{u(s,i(s))}} \right)ds} + 2 \!\!\!\!\!\!\!\!\!\!\!\! \sum_{\scriptsize{\begin{array}{c}(t^\prime,j^\prime) \in \Gamma(u),\\(0,0) \preceq (t^{\prime},j^{\prime}) \prec (t,j)\end{array}}} \!\!\!\!\!\!\!\!\!\!\!\! \sigma (\abs{u(t^\prime,j^\prime)}) 
 & \nonumber
\end{align}
for all $(t,j) \in \mathrm{dom}\,x$. Exploiting (\ref{eq:e3}) and denoting $\tilde\beta (\cdot,\cdot,\cdot) $ $ :=\beta (\alpha_2 (\cdot),\cdot,\cdot)$, $\gamma_1 (\cdot) := 2\sigma (\cdot)$ and $\gamma_2 (\cdot) := 2\sigma (\cdot), \alpha (\cdot) := \alpha_1(\cdot)$ gives the conclusion
\begin{align*} 
\alpha ( \omega (x(t,j)) ) \leq & \tilde{\beta} \left( \omega (x(0,0)) ,t,j \right) + \int_0^t {\gamma_1 \left( {\abs{u(s,i(s))}} \right)\mathrm{d}\,s } + 
\!\!\!\!\!\!\!\!\!\!\!\! \sum_{\scriptsize{\begin{array}{c}(t^\prime,j^\prime) \in \Gamma(u),\\(0,0) \preceq (t^{\prime},j^{\prime}) \prec (t,j)\end{array}}} \!\!\!\!\!\!\!\!\!\!\!\!\gamma_2 (\abs{u(t^\prime,j^\prime)}) . &
\end{align*}
\subsection{Proof of the implication $(\ref{item:iiss}) \Rightarrow (\ref{item:lyapunov-iiss})$} \label{S:2}

The proof is split into the following steps: 1) we recall Theorem~\ref{P:02} that an inflated system, say $\H_\sigma$, remains pre-iISS under small enough perturbations when $\H$ is pre-iISS; 2) we define an auxiliary system, say $\hat\H$, and then we show that some selection result holds for $\hat\H$ and $\H$; 3) we start constructing a smooth converse iISS Lyapunov function for $\H$ with providing a preliminary possibly non-smooth function, denoted by $V_0$, and we show that $V_0$ cannot increase too fast along solutions of $\hat\H$ (cf. Lemma~\ref{L:15} below); 4) we initially smooth $V_0$ and obtain the partially smooth function $V_s$ (cf. Lemma~\ref{L:12} below); 5) we smooth $V_s$ on the whole state space and get the smooth function $V_1$ (cf. Lemma~\ref{L:13} below); 6) we pass from the results for $\hat\H$ to the similar ones for $\H$ (cf. Lemma~\ref{L:14} below); 7) we give a characterization of 0-input pre-AS (cf. Proposition~\ref{P:2} below); 8) finally we combine the results of Lemma~\ref{L:14} with those of Proposition~\ref{P:2} to obtain the smooth converse iISS Lyapunov function $V$.
\begin{remark}
It should be noted that the construction of a smooth converse iISS Lyapunov function follows the same steps as those in \cite{Angeli.2000} but with different tools and technicalities.
Particularly, the authors in \cite{Angeli.2000} provided a preliminary possibly non-smooth iISS Lyapunov function and then appealed to~\cite[Theorem B.1]{Lin.1996} and~\cite[Proposition 4.2]{Lin.1996} to smooth the preliminary iISS Lyapunov function regardless robustness of iISS to sufficiently small perturbations. However, such a procedure does not necessarily hold for the case of hybrid systems as the procedure relies on uniform convergence of solutions. This is the reason that we appeal to results in \cite[Sections VI.B-C]{Cai.2007}, that is originally developed~\cite{Teel.2000}, to smooth our preliminary iISS Lyapunov function. Toward this end, we need to establish robustness of the pre-iISS property for hybrid systems to vanishing perturbations, which is challenging and has not been previously studied in the literature.
\hfill$\Box$
\end{remark}
\subsubsection{Robustness of pre-iISS}
Here we show robustness of pre-iISS to small enough perturbations (cf. Theorem~\ref{P:02} below). To be more precise, there exists an inflated hybrid system, denoted by $\H_\sigma$, remaining pre-iISS under sufficiently small perturbations when the original system $\H$ is pre-iISS. 

Given the hybrid system $\H$, a compact set $\A \subset \X$, and a continuous function $\sigma \colon \X \to \Rp$ that is positive on $\X \backslash \A$, the $\sigma$-perturbation of $\H$, denoted by $\H_\sigma$, is defined by
\begin{align}
& \H_\sigma := \left\{ \begin{array}{lccl}
\dot{\overline{x}} & \in & f_\sigma (\overline{x},u) & \quad (\overline{x},u) \in \C_\sigma \\
  \overline{x}^+ & \in & g_\sigma (\overline{x},u) & \quad (\overline{x},u) \in \D_\sigma
\end{array} \right. & \label{eq:er6}
\end{align}
where
\begin{eqnarray}
f_\sigma (\overline{x},u) & := & \overline{\mathrm{co}} f \left( (\overline{x} + \sigma (\overline{x}) \overline{\mathbb{B}},u) \cap \C \right) + \sigma(\overline{x}) \overline{\mathbb{B}} , \\
g_\sigma (\overline{x},u) & := &   \big\{ z \in \X : z \in v + \sigma (v) \overline{\mathbb{B}} , v \in g \left( (\overline{x} + \sigma (\overline{x}) \overline{\mathbb{B}},u) \cap \D \right) \big\} , \\
\C_\sigma & := & \left\{ (\overline{x},u) \colon (\overline{x}+\sigma (\overline{x}) \overline{\mathbb{B}},u) \cap \mathcal{C} \neq \emptyset \right\} , \\
\D_\sigma & := & \left\{ (\overline{x},u) \colon (\overline{x}+\sigma (\overline{x}) \overline{\mathbb{B}},u) \cap \mathcal{D} \neq \emptyset \right\} .
\end{eqnarray}
In what follows, by an \emph{admissible perturbation radius}, we mean any continuous function $\sigma \colon \X \to \Rp$ such that $x+\sigma(x) \overline{\mathbb{B}} \subset \X$ for all $x \in \X$.
\begin{theorem} \label{P:02}
Let $\H$ satisfy the Standing Assumptions.
Let $\A \subset \X$ be a compact set.
Assume that the hybrid system $\H$ is pre-iISS with respect to $\A$.
There exists an admissible perturbation radius $\sigma \colon \X \to \Rp$ that is positive on $\X \backslash \A$ such that the hybrid system $\H_\sigma$, the $\sigma$-perturbation of $\H$, is pre-iISS with respect to $\A$, as well.
\end{theorem}
\begin{proof} See Appendix~\ref{Ap:H}. \end{proof}
\begin{remark}
Besides the contribution of Theorem~\ref{P:02} to proof of our main result, it is of independent interest. We note that model \eqref{eq:er6} arises in many practical cases. For instance, assume that $\H$ is pre-iISS. Different types of perturbations such as slowly varying parameters, singular perturbations, highly oscillatory signals to $\H$ provide a perturbed system which may be modeled by \eqref{eq:er6} (cf. \cite{Wang.2012,Angeli.2002,Moreau.2001} for more details). Theorem \ref{P:02} guarantees pre-iISS of the perturbed system under the certain conditions.
\hfill$\Box$
\end{remark}
\subsubsection{The auxiliary system $\hat{\H}$ and the associated properties} 
We need to define the following auxiliary system $\hat{\H}$. Assume that $\H$ is pre-iISS with respect to $\A$ satisfying~\eqref{eq:e1} with suitable functions~$\alpha$, $\tilde\beta$, $\gamma_1$, $\gamma_2$. Pick any $\varphi \in \Kinf$ with $\max \{ \gamma_1 \circ \varphi(s), \gamma_2 \circ \varphi(s) \} \leq \alpha(s)$ for all $s \in \Rp$. Define the following hybrid inclusion
\begin{align}
& \hat{\H} := \left\{ \begin{array}{lccl}
\dot{x} & \in & \hat{F} (x) & \qquad x \in \hat{\C} \\
 x^+ & \in & \hat{G} (x) & \qquad x \in \hat{\D}
\end{array} \right. & \label{eq:er7}
\end{align}
where
\begin{eqnarray} \label{eq:er9}
\begin{array}{rcl}
\hat{F} (x) & := & \left\{ \nu \in \Rn \colon \nu \in f \left( x , u \right) , u \in \mathcal{U} \cap \varphi(\omega(x)) \overline{\mathbb{B}} \textrm{ and } (x,u) \in \mathcal{C} \right\} , \\
\hat{G} (x) & := & \left\{ \nu \in \X \colon \nu \in g \left( x , u \right) , u \in \mathcal{U} \cap \varphi(\omega(x)) \overline{\mathbb{B}} \textrm{ and } (x,u) \in \mathcal{D} \right\} , \\
\hat\C & := & \left\{ x \in \X \colon \exists \, u \in \mathcal{U} \cap \varphi(\omega(x)) \, \overline{\mathbb{B}} \textrm{ such that } (x,u) \in \mathcal{C} \right\} , \\
\hat\D & := & \left\{ x \in \X \colon \exists \, u \in \mathcal{U} \cap \varphi(\omega(x)) \, \overline{\mathbb{B}} \textrm{ such that } (x,u) \in \mathcal{D} \right\} . 
\end{array}
\end{eqnarray}
The hybrid inclusion (\ref{eq:er7}) is denoted by $\hat{\mathcal{H}} := (\hat{F}, \hat{G}, \hat\C , \hat\D,\mathcal{O})$ where $\mathcal{O} = \hat\C \cup \hat\D$.
We note that $\mathcal{O} = \X$ because $\X \supset \mathcal{O} = \hat\C \cup \hat\D \supset \Pi_0 (\C) \cup \Pi_0 (\D) = \X$.
We also note that $\hat{F}(x) = \overline{\mathrm{co}} \, \hat{F}(x)$ for each $x \in \hat\C$ and the data of $\hat\H$ satisfies the Hybrid Basic Conditions (cf. Assumption 6.5 in \cite{Goebel.2012}).
To distinguish maximal solutions to $\hat\H$ from those to $\H$, we denote a maximal solution to $\hat\H$ starting from $\xi$ by $x_\varphi (\cdot,\cdot,\xi)$. Let $\hat\varrho (\xi)$ denote the set of all maximal solutions of $\hat\H$ starting from $\xi \in \X$.

We first relate solutions to $\H$ to those to $\hat\H$ using the following claim whose proof follows from similar lines as in the proof of \cite[Claim 3.7]{Cai.2009} with minor modifications. 
\begin{claim} \label{Co:1}
Assume that $\H$ is pre-forward complete. For each solution $x$ to $\hat\H$, there exists a hybrid input $u$ such that $(x,u)$ is a solution pair to $\H$ with $\abs{u(t,j)} \leq \varphi(\omega(x(t,j)))$ for all $(t,j) \in \dom x$.
$\textrm{ }$\hfill$\Box$
\end{claim}
The following lemma assures that $\hat\H$ is pre-forward complete.
\begin{lemma} \label{L:11}
Pre-iISS of $\H$ implies that there exists $\varphi \in\Kinf$ such that $\hat{\mathcal{H}}$ is pre-forward complete.
\end{lemma}
\begin{proof} Let $d \colon \mathrm{dom} \, d \to \overline{\mathbb{B}}$ be a hybrid input with $\mathrm{dom} \, d = \mathrm{dom} \, x$ such that $d \in \mathcal{M}$, where 
\begin{align}
\mathcal{M} := \Big\{ & d \in \overline{\mathbb{B}} \colon \Big(x(t,j),\varphi \big(\omega (x(t,j))\big) d(t,j)\Big) \in \C \cup \D  \quad \forall (t,j) \in \mathrm{dom}\, x \Big\} . &\nonumber
\end{align}
By the definition of $\hat{\H}$, Claim~\ref{Co:1}, the pre-iISS assumption of $\mathcal{H}$ and the fact that $\max \{ \gamma_1 \circ\varphi(s), \gamma_2\circ\varphi(s) \} \leq \alpha (s)$ for all $s \in \mathbb{R}_{\geq 0}$, for each solution $x_\varphi$ to $\hat{\mathcal{H}}$, there exists a solution pair $(x_\varphi,\varphi(\omega(x_{\varphi}))d)$ to $\H$ with $d \in \mathcal{M}$ such that the following holds
\begin{align}
\alpha(\omega(x_\varphi (t,j,\xi))) \leq & \tilde{\beta} (\omega(\xi),t,j) + \int_0^t \gamma_{1}(\abs{d(s,i(s))} \varphi(\omega (x_{\varphi}(s,i(s),\xi)) )) \textmd{d}s & \nonumber \\
& + \!\!\!\!\!\!\!\!\!\!\!\!\!\!\!\sum_{\scriptsize{\begin{array}{c}(t^\prime,j^\prime) \in \Gamma(x_{\varphi}),\\(0,0) \preceq (t^{\prime},j^{\prime}) \prec (t,j)\end{array}}} \!\!\!\!\!\!\!\!\!\!\!\!\!\!\!\gamma_2 (\abs{d(t^\prime,j^\prime)} \varphi(\omega (x_{\varphi}(t^\prime,j^\prime,\xi)) )) & \nonumber\\
\leq & \tilde{\beta}_0 (\omega(\xi)) + \int_0^t \alpha(\omega (x_\varphi(s,i(s),\xi)) ) \mathrm{d}s + \!\!\!\!\!\!\!\!\!\!\!\!\!\!\!\sum_{\scriptsize{\begin{array}{c}(t^\prime,j^\prime) \in \Gamma(x_{\varphi}),\\(0,0) \preceq (t^{\prime},j^{\prime}) \prec (t,j)\end{array}}} \!\!\!\!\!\!\!\!\!\!\!\!\!\!\!\alpha (\omega (x_{\varphi}(t^\prime,j^\prime,\xi))) \label{eq:e10}
\end{align}
where $\tilde{\beta}_0 (\cdot) := \tilde{\beta}(\cdot,0,0)$. It follows with \cite[Proposition 1]{Noroozi.2014} that
$$
\alpha(\omega(x_{\varphi}(t,j,\xi))) \leq \tilde{\beta}_0 (\omega(\xi)) e^{t+j} \qquad \forall (t,j) \in \mathrm{dom} \, x .
$$
Therefore, the maximal solution $x$ is bounded if the corresponding hybrid domain is compact. It shows that every maximal solution of $x$ is either bounded or complete.
\end{proof}
The following hybrid inclusion is defined by
\begin{align*}
& \hat\H_\sigma := \left\{ \begin{array}{lccl}
\dot{\overline{x}} & \in & \hat{F}_\sigma (\overline{x}) & \quad \overline{x} \in \hat\C_\sigma \\
 \overline{x}^+ & \in & \hat{G}_\sigma (\overline{x}) & \quad \overline{x} \in \hat\D_\sigma
\end{array} \right. &
\end{align*}
where
\begin{eqnarray}
\hat{F}_\sigma (\overline{x}) & := & \big\{ \nu \in \Rn \colon \nu \in f_\sigma \left( \overline{x} , u \right) , u \in \mathcal{U} \cap \varphi(\omega(\overline{x})) \overline{\mathbb{B}} \;\textrm{and}\; (\overline{x},u) \in \C_\sigma \big\} , \nonumber \\
\hat{G}_\sigma (\overline{x}) & := & \big\{ \nu \in \X \colon \nu \in g_\sigma \left( \overline{x} , u \right) , u \in \mathcal{U} \cap \varphi(\omega(\overline{x})) \overline{\mathbb{B}} \;\textrm{and}\; (\overline{x},u) \in \mathcal{D}_{\sigma} \big\} , \nonumber \\
\hat\C_\sigma & := & \big\{ \overline{x} \in \X \colon \exists u \in \mathcal{U} \cap \varphi(\omega(\overline{x})) \textrm{ such that } (\overline{x},u) \in \C_\sigma \big\} , \nonumber \\
\hat\D_\sigma & := & \big\{ \overline{x} \in \X \colon \exists u \in \mathcal{U} \cap \varphi(\omega(\overline{x})) \textrm{ such that } (\overline{x},u) \in \D_\sigma \big\} . \nonumber
\end{eqnarray}
that is extended from $\hat{\H}$. We denote $\hat{\H}_\sigma$ by $(\hat{F}_\sigma,\hat{G}_\sigma,\hat\C_\sigma,\hat\D_\sigma,\X)$.
Since $\sigma$ is an admissible perturbation radius, $\hat\C_\sigma \cup \hat\D_\sigma = \hat\C \cup\hat\D$.
A maximal solution to $\hat{\mathcal{H}}_\sigma$ starting from $\overline\xi$ is denoted by $\overline x_\varphi (\cdot,\cdot,\xi)$.
Let $\hat{\varrho}_\sigma (\xi)$ denote the set of all maximal solution to $\hat{\H}_\sigma$ starting from $\xi \in \X$.
It is straightforward to see the combination of Lemma~\ref{L:11} and Theorem~\ref{P:02} ensures that $\hat{\mathcal{H}}_{\sigma}$ is pre-forward complete.
\begin{corollary} \label{C:3}
Pre-iISS of $\H_\sigma$ implies that there exists $\varphi\in\Kinf$ such that $\hat\H_\sigma$ is pre-forward complete.
\hfill$\Box$
\end{corollary}
It should be pointed out that, by \cite[Proposition 3.1]{Cai.2007}, $\hat\H_\sigma$ satisfies the Standing Assumptions as long as $\hat{\mathcal{H}}$ satisfies the same conditions and $\sigma$ is an admissible perturbation radius.

\subsubsection{The preliminary function $V_0$} 
We start constructing the smooth converse iISS Lyapunov function with giving a possibly nonsmooth function $V_0$. Before proceeding to the main result of this subsection, we define the following set. Consider a hybrid signal $d \colon \mathrm{dom} \, d \to \overline{\mathbb{B}}$ with $\mathrm{dom} \, d = \mathrm{dom} \, \overline{x}$ such that $d \in \overline{\mathcal{M}}$, where 
$$
\overline{\mathcal{M}} := \big\{ d \in \overline{\mathbb{B}} \colon (\overline{x}(t,j),\varphi (\omega (\overline{x}(t,j))) d(t,j)) \in \mathcal{C}_\sigma \cup \mathcal{D}_{\sigma} \quad \forall (t,j) \in \mathrm{dom}\, \overline{x}  \big \} .
$$
\begin{lemma} \label{L:15}
Let $\mathcal{A} \subset \X$ be a compact set. Also, let $\sigma \colon \X \to \Rp$ be an admissible perturbation radius that is positive on $\X \backslash \A$.
Let $\omega$ be a proper indicator on $\X$ for $\A$. Assume that $\H_\sigma$ is pre-iISS with respect to $\A$ satisfying~\eqref{eq:e1} with suitable functions~$\alpha \in \Kinf$,~$\overline{\beta} \in \KLL$, $\overline{\gamma}_1,\overline{\gamma}_2 \in \K$.
Let $\varphi  \in \Kinf$ such that $\max \{ \overline{\gamma}_1 \circ\varphi(s), \overline{\gamma}_2 \circ\varphi(s) \} \leq \overline{\alpha}(s)$ for all $s \in \Rp$.
Then there exists a function $V_0 \colon \X \to \Rp$ defined by
\begin{align}
V_0 (\xi) = \sup \big\{ z(t,j,\xi,d) \colon (t,j) \in \mathrm{dom} \, \overline{x}_{\varphi} \, , \, d \in \overline{\mathcal{M}} \big\} \label{eq:e53}
\end{align}
where for each $\xi \in \X$ and $d \in \overline{\mathcal{M}}$, $z(\cdot,\cdot,\xi,d)$ is defined by
\begin{align}
z(t,j,\xi,d) := & \alpha( \omega ( \overline{x}_{\varphi} (t,j,\xi) ) ) - \int_0^t \overline{\gamma}_1 (\abs{d(s,i(s))} \, \varphi ( \omega ( \overline{x}_\varphi (s,i(s),\xi) ) ) ) \mathrm{d} s  & \nonumber \\
&- \!\!\!\!\!\!\!\!\!\!\!\!\!\!\! \sum_{\scriptsize{\begin{array}{c}(t^\prime,j^\prime) \in \Gamma(\overline{x}_{\varphi}),\\(0,0) \preceq (t^{\prime},j^{\prime}) \prec (t,j)\end{array}}} \!\!\!\!\!\!\!\!\!\!\!\!\!\!\! \overline\gamma_2 (\abs{d(t^\prime,j^\prime)} \varphi(\omega (\overline{x}_{\varphi}(t^\prime,j^\prime,\xi)) )) \label{eq:e54}
\end{align}
such that
\begin{align}
& \alpha ( \omega (\xi) )  \leq V_0 (\xi) \leq \overline\beta_0 (\omega (\xi)) \quad \quad \forall \xi \in \X , \,\mathrm{ and } \;\,\overline\beta_0 (\cdot) := \overline\beta (\cdot,0,0) , & \label{eq:e56} \\
& V_0 (\overline{x}_{\varphi} (h,0,\xi)) - V_0 (\xi) \leq \int_{0}^{h} \overline{\gamma}_1 (\abs{\mu} \, \varphi ( \omega ( \overline{x}_{\varphi} (s,0,\xi) ) ) ) \mathrm{d} s \nonumber\\
&\qquad\qquad\qquad\qquad \forall \xi \in \hat{\C} \backslash \A , \abs{\mu} \leq 1, \overline{x}_\varphi \in \hat{\varrho}_\sigma (\xi) \,\mathrm{ with }\, (h,0) \in \dom \, \overline{x}_\varphi , & \label{eq:e57} \\
& V_0 (g) - V_0 (\xi) \leq \overline{\gamma}_2 (\abs{\mu} \varphi (\omega(\xi))) \;\; \forall \xi \in \hat{\D} , g \in \hat{G}(\xi) , \abs{\mu} \leq 1 . & \label{eq:e58}
\end{align}
\end{lemma}
\begin{proof}
See Appendix~\ref{Ap:K}.
\end{proof}
\subsubsection{Initial smoothing}
Here we construct a partially smooth function on $\X$ from $V_{0}$.
\begin{lemma} \label{L:12}
Let $\A \subset \X$ be a compact set. Also, let $\sigma \colon \X \to \Rp$ be an admissible perturbation radius that is positive on $\X \backslash \A$. Let $\omega$ be a proper indicator on $\X$ for $\A$.
Assume that $\H_\sigma$ is pre-iISS with respect to $\A$.
Then for any $\xi \in \X$ and $\abs{\mu} \leq 1$, there exist $\underline\alpha_s, \overline\alpha_s, \tilde\gamma_1 , \tilde\gamma_2 \in \Kinf$, and a continuous function $V_s \colon \X \to \Rp$, smooth on $\X \backslash \A$, such that
\begin{eqnarray*}
\underline\alpha_s ( \omega (\xi) )  \leq V_s (\xi) &\leq& \overline{\alpha}_{s} (\omega (\xi)) \qquad \;\;\,\quad \forall \xi \in \X ,  \\
\max_{f \in \hat{F}(\xi)} \langle \nabla V_{s} (\xi),f \rangle &\leq& \tilde{\gamma}_1 (\abs{\mu} \varphi (\omega(\xi))) \quad \forall \xi \in \hat{\C} \backslash \mathcal{A} ,  \\
\max_{g \in \hat{G}(\xi)} V_s (g) - V_s (\xi) &\leq& \tilde{\gamma}_{2} (\abs{\mu} \varphi (\omega(\xi))) \quad \forall \xi \in \hat{\D} . 
\end{eqnarray*}
\end{lemma}
\begin{proof} Let the functions $V_0$, $\alpha$, $\overline\beta$, $\overline\gamma_1$, $\overline\gamma_2$ and $\varphi$ come from Lemma~\ref{L:15}.
We begin with giving the following property of $V_0$ whose proof follows from the similar arguments as those in \cite[Proposition 7.1]{Cai.2007} with essential modifications. 
\begin{proposition} \label{P:3}
The function $V_0$ is upper semi-continuous on $\X$.
\hfill$\Box$
\end{proposition}
To prove the lemma, we follow the same approach as the one in \cite[Section VI.B]{Cai.2007} to construct a partially smooth function $V_s$ from $V_0$.
Let $\psi \colon \Rn \rightarrow [0,1]$ be a smooth function which vanishes outside of $\overline{\mathbb{B}}$ satisfying $\int \psi (\xi) \mathrm{d} \xi = 1$ where the integration (throughout this subsection) is over $\Rn$.
We find a partially smooth and sufficiently small function $\tilde{\sigma} \colon \X \backslash \A \rightarrow \Rsp$ and define the function $V_s \colon \X \rightarrow \Rp$ by
\begin{eqnarray} \label{eq:e89}
V_s (\xi) := \left\{ \begin{array}{l}
 0 \quad \quad \quad \quad \quad \quad \qquad \qquad \;\; \mathrm{for} \; \xi \in \A , \\ 
 \int V_{0} (\xi + \tilde{\sigma} (\xi) \eta ) \psi (\eta) \mathrm{d} \eta \quad \; \mathrm{for} \; \xi \in \X \backslash \A . \\ 
 \end{array} \right.
\end{eqnarray}
so that some desired properties (cf. items (a), (b) and (c) below) are met.
In other words, we find an appropriate $\tilde\sigma$ such that the following are obtained
\begin{itemize}
    \item [(a)] The function $V_s$ is well-defined, continuous on $\X$, smooth and positive on $\X \backslash \A$;
\item[(b)] as much as possible for some $\underline\alpha_s,\overline\alpha_s \in \Kinf$ the following conditions hold
\begin{align}
& V_s (\xi) |_{\xi \in \A} = 0, & \label{eq:e88} \\
& \underline\alpha_s (\omega(\xi)) \leq V_s (\xi) \leq \overline\alpha_s (\omega(\xi)) \qquad \forall \xi \in \X ; \label{eq:e2} &
\end{align}
	 \item[(c)] for some $\tilde{\gamma}_1,\tilde{\gamma}_2 \in \Kinf$, it holds that
\begin{align}
\max_{f \in \hat{F}(\xi)} \langle \nabla V_s (\xi),f \rangle & \leq \tilde{\gamma}_1 (\abs{\mu} \varphi (\omega(\xi)) ) \qquad \forall \xi \in \hat{\C} \backslash \A , & \label{eq:e90} \\
\max_{g \in \hat{G}(\xi)} V_s (g) - V_s (\xi) & \leq \tilde{\gamma}_{2} (\abs{\mu} \varphi (\omega(\xi)))  \qquad \forall \xi \in \hat{\D} . \label{eq:e91} &
\end{align}
\end{itemize}
Regarding (a), we appeal to \cite[Theorem 3.1]{Kellett.2005} to achieve the desired properties. This theorem requires that $V_0 (\xi) |_{\xi \in \A} = 0$, which is shown in the previous subsection, $V_0$ is upper semi-continuous on $\X$, which is established by Proposition~\ref{P:3}, and the openness of $\X \backslash \A$, which is guaranteed by \cite[Lemma 7.5]{Cai.2007}.

Regarding (b), the property (\ref{eq:e88}) follows from the definition of $V_s$, the upper semi-continuity of $V_0$, and the openness of $\X \backslash \mathcal{A}$. Also, it follows from \cite[Lemma 7.7]{Cai.2007} that we can pick the function $\tilde\sigma$ sufficiently small such that for any $\mu_1,\mu_2 \in \Kinf$ satisfying
\begin{align}
& \mu_1 (s) < s < \mu_2 (s) \qquad \forall s \in \Rsp , &
\end{align}
the following holds
\begin{align}
& \alpha(\mu_1 (\omega(\xi))) < V_s (\xi) < \overline\beta_0 (\mu_2 (\omega(\xi))) \qquad \forall \xi \in \X . & \label{eq:e96}
\end{align}
So the inequalities (\ref{eq:e2}) are obtained, as well.

Regarding (c), let $\sigma_2$ be a continuous function that is positive on ${\X} \backslash \A$ and that satisfies $\sigma_2 (\xi) \leq \sigma(\xi) $ for all $\xi \in {\X}$.
We first construct functions $\sigma_2$ and $\tilde{\sigma}$ so that for each $\xi \in {\X}\backslash \mathcal{A}$, for each $\overline{x}_\varphi \in \hat{\varrho}_{\sigma_2} (\xi)$, for each $\eta \in \overline{\mathbb{B}}$ and $(t,j) \in \mathrm{dom}\,\overline{x}_\varphi$ such that $\overline{x}_\varphi (t,j,\xi) \in \X \backslash \A$, the function defined on $(t,j) \in \dom\,\overline{x}_\varphi \cap [0,t] \times \{0,\dots,j\}$ given by $(\tau,k) \mapsto \overline{x}_\varphi (\tau,k) + \tilde{\sigma} (\overline{x}_\varphi (\tau,k)) \eta$ can be extended to a complete solution of $\hat{\H}_\sigma$.
Now, pick a maximal solution $\overline{x}_\varphi (h,m,\xi)$ to $\hat{\H}_{\sigma_2}$. First, let $m = 0$.
So according to the definition of $V_s$, Lemma 7.2 in \cite{Cai.2007}, \eqref{eq:e57} and the fact $\psi \colon \Rn \rightarrow [0,1]$ that we get for any $\abs{\mu} \leq 1$ and for any $\overline{x}_{\varphi} \in \hat{\varrho}_{\sigma_2} (\xi)$ so that $\xi \in \hat{\C} \backslash \A$
\begin{align}
V_s (\overline{x}_\varphi (h,0,\xi)) \leq \! V_s ( \xi ) + \!\!\int\! \Big\{\! \int_0^h \overline{\gamma}_1\! (\abs{\mu}\varphi(\omega(\overline{x}_\varphi (s,0,\xi) + \tilde{\sigma} (\overline{x}_\varphi (s,0,\xi))\eta))) \mathrm{d} s \Big\} \psi (\eta) \mathrm{d} \eta . & \label{eq:e9}
\end{align}
It follows from \cite[Claim 6.3]{Cai.2007} that for any $\xi \in \hat{\C} \backslash \A$ and $f \in \hat{F}(\xi)$, there exists a solution $\overline{x}_\varphi \in \hat{\varrho}_{\sigma_2} (\xi)$ such that for small enough $h > 0$, we get that $(h,0) \in \mathrm{dom} \, \overline{x}_\varphi$ and $\overline{x}_\varphi = \xi + h f$. So it follows with smoothness of $V_s$ on $\X \backslash \A$, Claim 6.3 in \cite{Cai.2007}, the inequality (\ref{eq:e9}) and the mean-value theorem that
\begin{align}
\left\langle {\nabla V_s , f} \right\rangle = & \lim_{h \to 0^{+}} \frac{V_s (\xi + h f) - V_s (\xi)}{h} & \nonumber \\
\leq & \lim_{h \to 0^+} \int \overline{\gamma}_1 (\abs{\mu}\varphi(\omega(z + \tilde{\sigma} (z)\eta ))) \psi (\eta) \mathrm{d} \eta . & \nonumber
\end{align}
where $z$ lies in the line segment joining $\xi$ to $\xi + h f$. It follows from uniform continuity of $\omega$ with respect to $\eta$ on $\overline{\mathbb{B}}$ that for any $\xi \in \hat{\C} \backslash \A$ and $f \in \hat{F}(\xi)$
\begin{align}
\left\langle {\nabla V_s , f} \right\rangle \leq & \int \overline{\gamma}_1 (\abs{\mu}\varphi(\omega(\xi + \tilde{\sigma} (\xi)\eta))) \psi (\eta) \mathrm{d} \eta & \nonumber \\
\leq & \sup_{z \in \xi + \tilde{\sigma} (\xi) \overline{\mathbb{B}}} \overline{\gamma}_1 (\abs{\mu}\varphi(\omega(z))) . & \nonumber
\end{align}
From Claim 7.6 and Lemma 7.7 in \cite{Cai.2007}, there exists some $\sigma_u (\cdot)$ with $\tilde{\sigma} (\xi) \leq \sigma_{u} (\xi)$ for all $\xi \in {\X} \backslash \A$ so that we get for all $\xi \in \hat{\C} \backslash \mathcal{A}$ and $f \in \hat{F}(\xi)$
\begin{align}
\left\langle {\nabla V_s , f} \right\rangle \leq  & \sup_{z \in \xi + \sigma_u (\xi) \overline{\mathbb{B}}} \overline{\gamma}_1 (\abs{\mu}\varphi(\omega(z))) & \nonumber \\
\leq &  \overline{\gamma}_1 (\abs{\mu} \varphi(\mu_2 (\omega(\xi)))) . \label{eq:e100}
\end{align}
Therefore, it is easy to see that for any $\overline\gamma_1 , \varphi , \mu \in \Kinf$ with $\mu_2 > \id$ and any $\abs{\mu} \leq 1$ the exists $\tilde\gamma_1 \in \Kinf$ such that~\eqref{eq:e90} holds.

Now let $(h,m) = (0,1)$. So it follows with the definition of $V_s$, Lemma 7.2 in \cite{Cai.2007}, the growth condition \eqref{eq:e58}, and the fact that $\psi \colon \Rn \rightarrow [0,1]$ that for any $\abs{\mu} \leq 1$ and each $\xi \in \hat{\mathcal{D}}$ and $g \in \hat{G}(\xi)$
\begin{align}
V_{s} (x_{\varphi} (0,1,\xi)) & \leq  V_s (\xi) + \int \overline{\gamma}_{2} (\abs{\mu}\varphi(\omega(\xi + \tilde{\sigma} (\xi) \eta))) \psi (\eta) \mathrm{d} \eta & \nonumber \\
&\leq  V_{s} (\xi) + \sup_{z \in \xi + \tilde{\sigma} (\xi) \overline{\mathbb{B}}} \overline{\gamma}_{2} (\abs{\mu}\varphi(\omega(z))) . & \nonumber
\end{align}
From \cite[Claim 7.6]{Cai.2007} and \cite[Lemma 7.7]{Cai.2007}, there exists $\sigma_{u}$ with $\tilde{\sigma} (\xi) \leq \sigma_{u} (\xi)$ for all $\xi \in {\X} \backslash \mathcal{A}$ so that we have for all $\xi \in \mathcal{D} \backslash \mathcal{A}$ and $g \in \hat{G}(\xi)$
\begin{align}
V_{s} (g) \leq & V_{s} (\xi) + \sup_{z \in \xi + \tilde{\sigma} (\xi) \overline{\mathbb{B}}} \overline{\gamma}_{2} (\abs{\mu}\varphi(\omega(z))) & \nonumber \\
\leq & V_{s} (\xi) + \sup_{z \in \xi + \sigma_{u} (\xi) \overline{\mathbb{B}}} \overline{\gamma}_{2} (\abs{\mu} \varphi( \omega(z))) & \nonumber \\
\leq & V_{s} (\xi) + \overline{\gamma}_{2} (\abs{\mu} \varphi(\mu_{2}(\omega(\xi)))) . & \label{eq:e95}
\end{align}
With the same arguments as those for flows, there exists $\tilde{\gamma}_2 \in \Kinf$ such that the following holds
$$
V_{s} (g) \leq  V_{s} (\xi) + \tilde{\gamma}_{2} (\abs{\mu} \varphi(\omega(\xi))) . 
$$
Moreover, if $\xi \in \hat{\mathcal{D}}$ and $g \in \mathcal{A}$ then $0 = V_{s} (g) \leq V_{s} (\xi) + \tilde{\gamma}_2 (\abs{\mu} \varphi(\omega(\xi)))$. So the growth condition (\ref{eq:e91}) holds.
\end{proof}
\subsubsection{Final smoothing}
The next lemma is to do with smoothing $V_{s}$ on $\mathcal{A}$.
\begin{lemma} \label{L:13}
Let $\mathcal{H}$ be pre-iISS. Also, let $V_{s}$, $\tilde{\gamma}_{1},\tilde{\gamma}_{2}$ and $\varphi$ come from Lemma~\ref{L:12}. For any $\xi \in {\X}$ and $\abs{\mu} \leq 1$, there exist $\underline{\alpha}, \overline{\alpha} \in \Kinf$, and a $\Kinf$-function $p$, smooth on $(0,+\infty)$ such that $V_1 \colon {\X} \to \Rp$ is defined by
\begin{align}
V_{1} (\xi) := p ( V_{s} (\xi)) \qquad \qquad \forall \xi \in {\X} 
\end{align}
where $V_s$, coming from Lemma~\ref{L:12}, is smooth on ${\X}$ and the following hold
\begin{eqnarray}
\underline{\alpha} ( \omega (\xi) ) \leq V_{1} (\xi) &\leq& \overline{\alpha} ( \omega (\xi) ) \;\;\;\;\,\!\qquad \qquad \forall \xi \in {\X} ,  \label{eq:e92} \\
\max_{f \in \hat{F}(\xi)} \langle \nabla V_{1} (\xi),f \rangle &\leq& \tilde{\gamma}_{1} (\abs{\mu} \varphi (\omega(\xi))) \qquad \forall \xi \in \hat{\mathcal{C}} , \label{eq:e93} \\
\max_{g \in \hat{G}(\xi)} V_{1} (g) - V_{1} (\xi) &\leq& \tilde{\gamma}_{2} (\abs{\mu} \varphi (\omega(\xi))) \qquad \forall \xi \in \hat{\mathcal{D}} . \label{eq:e94}
\end{eqnarray}
\end{lemma}
\begin{proof} With Lemma 4.3 in \cite{Lin.1996}, there exists a smooth function $p \in \Kinf$ such that $p' (s) > 0$ for all $s > 0$ where $p' (\cdot) := \frac{dp}{ds} (\cdot)$ and $p(V_{s} (\xi))$ is smooth for all $\xi \in {\X}$.
Without loss of generality, one can assume that $p' (s) \leq 1$ for all $s > 0$ (cf. Page 1090 of \cite{Angeli.2000} for more details).
Using the definition of $V_1$ and (\ref{eq:e96}), we have
\begin{align}
& p \circ \alpha \circ \mu_1 ( \omega (\xi) ) \leq V_1 (\xi) \leq p \circ \overline\beta_0 \circ \mu_2 ( \omega (\xi) ) \qquad \forall \xi \in \X . &
\end{align}
Therefore, (\ref{eq:e92}) holds.

It follows from, in succession, the definition of $V_1$, (\ref{eq:e90}) and the fact that $0 < p' (s) \leq 1$ for all $s > 0$ that for all $\xi \in \hat{\mathcal{C}} \backslash \mathcal{A}$
\begin{align*}
\max_{f \in \hat{F}(\xi)} \langle \nabla V_{1} (\xi),f \rangle & \leq p' (V_{2}) \tilde{\gamma}_{1} (\abs{\mu} \varphi (\omega(\xi))) \leq \tilde{\gamma}_{1} (\abs{\mu} \varphi (\omega(\xi))) .
\end{align*}
It follows with the fact that $\nabla V_1 (\xi) = 0$ and $\omega (\xi) = 0$ for all $\xi \in \A$, and $\tilde{\gamma}$ and $\varphi$ are zero at zero that
\begin{align*}
\max_{f \in \hat{F}(\xi)} \langle \nabla V_1 (\xi),f \rangle \leq \tilde{\gamma}_1 (\abs{\mu} \varphi (\omega(\xi))) \qquad \forall \xi \in \hat{\C} .
\end{align*}
It follows with, in succession, the definition of $V_1$, the mean-value theorem, the last inequality of (\ref{eq:e95}), the fact that $0 < p' (s) \leq 1$ for all $s > 0$ that for all $\xi \in \hat{\D}$
\begin{align*}
V_1 (g) - V_1 (\xi) = p'(z) (V_s (g) - V_s (\xi)) \leq \tilde{\gamma}_2 (\abs{\mu} \varphi (\omega(\xi))) 
\end{align*}
where $z$ lies on the segment joining $V_s (\xi)$ to $V_s (g)$.
\end{proof}
\subsubsection{Return to $\H$}
The following lemma is immediately obtained from Lemma~\ref{L:13} and \eqref{eq:er9}.
\begin{lemma} \label{L:14}
Let $\H$ be pre-iISS. Let $\varphi,\tilde{\gamma}_1,\tilde{\gamma}_2 \in \Kinf$ be generated by Lemma~\ref{L:12}. Also, let $\underline{\alpha} , \overline{\alpha} \in \Kinf$ and $V_1 \colon {\X} \to \Rp$ come from Lemma~\ref{L:13}. Then the following hold
\begin{align*}
& \underline{\alpha} ( \omega(\xi) ) \leq V_1 (\xi) \leq \bar{\alpha} ( \omega(\xi) ) \quad \forall\xi \in {\X} , &
\end{align*}
for any $(\xi,u) \in \C$ with $\abs{u} \leq \varphi (\omega(\xi))$
\begin{align*}
& \langle \nabla V_1 (\xi),f (\xi,u) \rangle \leq \tilde{\gamma}_1 (\abs{u} ) , &
\end{align*}
for any $(\xi,u) \in \D$ with $\abs{u} \leq \varphi (\omega(\xi))$
\begin{align*}
& V_1 (g(\xi,u)) - V_1 (\xi) \leq \tilde{\gamma}_2 ( \abs{u} ) . &
\end{align*}
\end{lemma}
\subsubsection{A characterization of 0-input pre-AS}
To continue with the proof, we need a dissipation characterization of 0-input pre-AS, which is stated in Proposition~\ref{P:2}. This proposition is a unification and generalization of \cite[Proposition II.5]{Angeli.2000}.
\begin{proposition} \label{P:2}
$\H$ is 0-input pre-AS if and only if there exist a smooth semi-proper function $W \colon {\X} \to \Rp$, $\lambda \in \K$ and a continuous function $\rho \in \mathcal{PD}$ such that
\begin{eqnarray}
\left\langle {\nabla W(\xi),f(\xi,u)} \right\rangle &\leq& - \rho ( \omega(\xi) ) + \lambda ( \abs{u} ) \quad \forall (\xi,u) \in \C , \label{eq:e17} \\
W ( g(\xi ,u) ) - W ( \xi ) & \leq & - \rho ( \omega (\xi ) ) + \lambda ( \abs{u} ) \quad \lambda (\xi,u) \in \D . \label{eq:e18}
\end{eqnarray}
\end{proposition}
\begin{proof}
Sufficiency is clear. We establish necessity. To this end, the following lemma is needed.
\begin{lemma} \label{L:5}
$\H$ is 0-input pAS if and only if there exist a smooth Lyapunov function $V \colon \X \to \Rp$ and $\alpha_1,\alpha_2,\alpha_3,\chi \in \Kinf$ and a nonzero smooth function $q \colon \Rp \to \Rsp$ with the property that $q(s) \equiv 1$ for all $s \in [0,1]$ such that
\begin{align}
\alpha_1 ( \omega (\xi) ) \leq V(\xi) &\leq \!\alpha_2 ( \omega (\xi) ) \;\quad\; \forall \,\, \xi \in {\X} , & \label{eq:e19} \\
\left\langle \nabla V(\xi),f(\xi,q(\omega(\xi))I\nu) \right\rangle &\leq \!\! - \alpha_3 \left( \omega ( \xi ) \right) \!\quad \forall  (\xi ,q(\omega(\xi))I\nu) \!\in\! \C \!\textrm{ with } \!\omega (\xi) \!> \!\chi (\abs{\nu}) , & \label{eq:e20} \\
V(g(\xi,q(\omega(\xi))I\nu)) - V(\xi) &\leq \!\! - \alpha_3 \left(\omega (\xi) \right) \!\quad \forall  (\xi,q(\omega(\xi))I\nu) \!\in\! \D \!\textrm{ with }\! \omega (\xi) \!>\! \chi (\abs{\nu}) . & \label{eq:e33}
\end{align}
where $I$ is the $m \times m$ identity matrix.
\end{lemma}
\begin{proof}
See Appendix~\ref{Ap:G}.
\end{proof}
Now we can pursue the proof of Proposition~\ref{P:2}. Let $\mathcal{H}$ be 0-input pre-AS. Recalling Lemma~\ref{L:5}, there exists a Lyapunov function $V$ with the properties (\ref{eq:e19})-(\ref{eq:e33}). Using \cite[Remark 2.4]{Sontag.1995}, we can show that there exists  some $\alpha_{4} \in \Kinf$ such that (\ref{eq:e20}) and (\ref{eq:e33}) are equivalent to
\begin{eqnarray}
 \left \langle \nabla V(\xi),f(\xi,q(\omega(\xi))I\nu) \right\rangle &\leq& - \alpha_3 \left( \omega (\xi) \right) + \alpha_{4} \left( \abs{\nu} \right) \qquad \forall (\xi,q(\omega(\xi))I\nu) \in \mathcal{C} , \label{eq:e22} \\
V(g(\xi,q(\omega(\xi))I\nu)) - V(\xi) &\leq&  - \alpha_3 \left( \omega (\xi) \right) + \alpha_4 \left( \abs{\nu} \right) \qquad \forall (\xi,q(\omega(\xi))I\nu) \in \mathcal{D} .  \label{eq:e30}
\end{eqnarray}
Given \cite[Corollary IV.5]{Angeli.2000}, there exists $\lambda \in \mathcal{K}$ such that $\alpha_4 (sr) \leq \lambda(s) \lambda(r)$ for all $(s,r) \in \Rp \times \Rp$. So we have
\begin{eqnarray*}
\left \langle \nabla V(\xi),f(\xi,u) \right\rangle &\leq& - \alpha_{3} (\omega (\xi)) + \lambda (1/q(\omega(\xi))) \lambda (\abs{u}) \qquad \; \forall (\xi,u) \in \C , \\
V(g(\xi,u)) - V(\xi) &\leq&  - \alpha_3 (\omega (\xi)) + \lambda (1/q(\omega(\xi))) \lambda (\abs{u}) \;\qquad \forall (\xi,u) \in \D  
\end{eqnarray*}
where $u := q(\omega(\xi))I\nu$. Define $\pi \colon \Rp \to \Rp$ as
\begin{equation*}
\pi ( r ) = \int_0^r {\frac{ds} {c + \theta (s)}}
\end{equation*}
where $c > 0$ and $\theta \in \mathcal{K}$ are defined below. We note that $\pi \in \K$. Let $W(r) := \pi (V(r))$ for all $r \geq 0$. Taking the time derivative and difference of $W(\xi)$ and recalling~(\ref{eq:e22}) and~(\ref{eq:e30}) yield
\begin{align*}
\left \langle {\nabla W(\xi ),f(\xi,u)} \right\rangle & \leq \frac{\left\langle {\nabla V(\xi ),f(\xi,u)} \right\rangle } {c + \theta ( V(\xi))} \nonumber \\
& \leq  - \frac{\alpha_3 ( \omega (\xi) )} {c + \theta  ( V(\xi))} + \frac{\lambda ( 1/q ( \omega (\xi) )) \lambda (\abs{u})} {c + \theta ( V(\xi))} \qquad \forall (\xi,u) \in \mathcal{C} , \\
W( g(\xi,u)) - W(\xi) & \leq \frac{V(g(\xi,u)) - V(\xi)} {c + \theta ( V(\xi))} \nonumber \\
& \leq - \frac{\alpha_3 (\omega (\xi))} {c + \theta ( V(\xi))} + \frac{\lambda ( 1/q ( \omega (\xi) ) ) \lambda (\abs{u})} {c + \theta ( V(\xi) )} \qquad \forall (\xi,u) \in \mathcal{D} .
\end{align*}
It follows from~(\ref{eq:e19}) that 
\begin{eqnarray*}
\left \langle {\nabla W(\xi),f(\xi,u)} \right\rangle &\leq& - \frac{\alpha_3 (\omega(\xi))} {c + \theta  \circ \alpha_2 (\omega(\xi))} + \frac{\lambda ( 1/q ( \omega (\xi)) ) \lambda (\abs{u})} {c + \theta \circ \alpha_1(\omega(\xi))} \qquad \forall (\xi,u) \in \C , \\
W( g(\xi,u)) - W(\xi) &\leq& - \frac{\alpha_{3} (\omega (\xi))} {c + \theta \circ \alpha_2 (\omega(\xi))} + \frac{\lambda ( 1/q ( \omega (\xi))) \lambda (\abs{u})} {c + \theta \circ \alpha_1 (\omega(\xi))} \qquad \forall (\xi,u) \in \D .
\end{eqnarray*}
Let $c := \lambda(g(0)) = \lambda(1)$. By the fact that $q$ is smooth everywhere and the definition of $c$, one can construct $\theta\in \mathcal{K}$ such that
\begin{align} \label{eq:e99}
c + \theta \circ \alpha_1 (s) \geq \lambda ( 1/q (s)) \qquad\qquad\qquad s \in \mathbb{R}_{\geq 0} .
\end{align}
It follows with (\ref{eq:e99}) that
\begin{eqnarray*}
\left \langle \nabla W(\xi),f(\xi,u) \right\rangle &\leq& - \rho (\omega(\xi)) + \lambda (\abs{u}) \qquad\qquad \forall (\xi,u) \in \C , \\
W( g(\xi,u)) - W(\xi) &\leq& - \rho (\omega (\xi)) + \lambda (\abs{u}) \qquad\qquad \forall (\xi,u) \in \D .
\end{eqnarray*}
where $\rho (s) := \frac{\alpha_3 (s)} {c + \theta \circ \alpha_2 (s)}$ for all $s \geq 0$. This proves the necessity.
\end{proof}
As pre-iISS implies 0-input pre-AS, it follows from Proposition~\ref{P:2} that there exist a smooth semi-proper function $W$, $\lambda \in \K$ and $\rho \in \mathcal{PD}$ such that (\ref{eq:e17}) and (\ref{eq:e18}) hold. Define $V \colon \X \to \Rp$ by $V (\xi) := W (\xi) + V_1 (\xi)$ with $V_1$ coming from Lemma~\ref{L:14}. It follows from Lemma~\ref{L:14} and Proposition~\ref{P:2} that $V$ is smooth everywhere and there exist $\alpha_1,\alpha_2 \in \Kinf$ such that
\begin{align}
\alpha_1 ( \omega (\xi) ) & \leq V (\xi) \leq \alpha_2 ( \omega (\xi) ) \qquad \qquad \forall \xi \in \X . & \label{eq:e48}
\end{align}
We also have for any $(\xi,u) \in \C$ with $\abs{u} \leq \varphi (\omega(\xi))$
\begin{align*}
\langle \nabla V (\xi),f (\xi,u) \rangle & \leq - \rho (\omega(\xi)) + \eta (\abs{u}), &
\end{align*}
and for any $(\xi,u) \in \D$ with $\abs{u} \leq \varphi (\omega(\xi))$
\begin{align*}
V (g(\xi,u)) - V (\xi) & \leq - \rho (\omega(\xi)) + \eta (\abs{u}) &
\end{align*}
where $\eta (\cdot) := \tilde{\gamma} (\cdot) + \lambda (\cdot)$ and $\tilde{\gamma} (\cdot) := \max \{ \tilde{\gamma}_1 (\cdot) , \tilde{\gamma}_2 (\cdot) \}$. To show that $V$ satisfies~\eqref{eq:e3} and~\eqref{eq:e4}, let $\chi = \varphi^{-1}$ and define
\begin{align*}
\hat{\kappa} (r) :=\!\! \max_{ \!\omega(\xi) \leq \chi(\abs{u}), \abs{u} \leq r, u \in \mathcal{U}} & \big\{ \left\langle {\nabla V(\xi) , f(\xi,u)} \right\rangle + \rho ( \omega (\xi) ), V (g(x,u)) - V (\xi) + \rho (\omega(\xi)) \!\big\} . & 
\end{align*}
Then
$$
\kappa (r) := \max \{ \hat{\kappa} (r) , \eta (r) \} .
$$
It is obvious that $\kappa \in \mathcal{K}$. By considering two cases of $u \in \mathcal{U}$ in which $\abs{u} \leq \varphi(\omega(\xi))$ and $\abs{u} \geq \varphi(\omega(\xi))$, we get
\begin{align*}
\left \langle {\nabla V(\xi),f(\xi,u)} \right\rangle & \leq - \rho (\omega(\xi)) + \kappa ( \abs{u} ) \quad \forall (\xi,u) \in \mathcal{C} , & \\
V (g(\xi,u)) - V (\xi) & \leq - \rho (\omega(\xi)) + \kappa ( \abs{u} ) \quad \forall (\xi,u) \in \mathcal{D} . &
\end{align*}
These estimates together with (\ref{eq:e48}) show that $V$ is a smooth iISS Lyapunov function for $\H$.

\section{iISS for Sampled-Data Systems}\label{sec:examples}
A popular approach to design sampled-data systems is the emulation approach. The idea is to first ignore communication constraints and design a continuous-time controller for a continuous-time plant. Then to provide certain conditions under which stability of the sampled-data control system in a certain sense is preserved in a digital implementation. The emulation approach enjoys considerable advantages in terms of the choice of continuous-time design tools. A central issue in the emulation design is the choice of the sampling period guaranteeing stability of the sampled-data system with the emulated controller. In a seminal work, Ne\v{s}i\'{c} et al. \cite{Nesic.2009} developed an explicit formula for a maximum allowable sampling period (MASP) that ensures asymptotic stability of sampled-data nonlinear systems with emulated controllers.

Here we show the effectiveness of Theorem~\ref{T:1} by establishing that the MASP developed in \cite{Nesic.2009} also guarantees iISS for a sampled-data control system. Consider the following plant model
\begin{align} \label{eq:e01s}
&\begin{array}{rcl}
\dot{x}_p &=& f_p (x_p,u,w) \\
y &=& g_p (x_p)
\end{array}
\end{align}
where $x_p \in \Rn[n_p]$ is the plant state, $u \in \Rn[n_u]$ is the control input, $w \in \Rn[n_w]$ is the disturbance input and $y \in \Rn[n_y]$ is the plant output. Assume that $f_p \colon \R^{n_p} \times \R^{n_u} \times \R^{n_w} \to \R^{n_p}$ is locally Lipschitz and $f_p (0,0) = 0$. Since we follow the emulation method, we assume that we know a continuous-time controller, which stabilizes the origin of system \eqref{eq:e01s} in the sense of iISS in the absence of network. We focus on dynamic controllers of the form
\begin{align} \label{eq:e02s}
\begin{array}{rcl}
\dot{x}_c & = &  f_c (x_c,y) \\
u & = & g_c (x_c)
\end{array} &
\end{align}
where $x_c \in \R^{n_c}$ is the controller state. Let $g_c \colon \Rn[n_c] \to \Rn[n_u]$ be continuously differentiable in its argument.

We consider the scenario where the plant and the controller are connected via a digital channel. In particular, we assume that the plant is between a hold device and a sampler. Transmissions occur only at some given time instants $t_j, j \in \Zsp$, such that $\epsilon \leq t_j-t_{j-1} \leq \tau_\mathrm{MASP}$, where $\epsilon\in(0,\tau_\mathrm{MASP}]$  represents the minimum time between any two transmission instants. Note that $\epsilon$ can be taken arbitrarily small and it is only used to prevent Zeno behavior \cite{Goebel.2012}. As in~\cite{Nesic.2009}, a sampled-data control system with an emulated controller of the form~\eqref{eq:e02s} can be modeled by
\begin{align} \label{eq:e09s}
&\begin{array}{rcll}
 \dot{x}_p &=& f_p (x_p,\hat{u},w) & t \in [t_{j-1},t_j] \\
 y &=& g_p (x_p) \\
 \dot{x}_c &=& f_c (x_c,\hat{y}) & t \in [t_{j-1},t_j] \\
 u &=& g_c (x_c,\hat{y}) \\
 \dot{\hat{y}} &=& \hat{f}_p (x_p,x_c,\hat y, \hat u) & t \in [t_{j-1},t_j] \\
 \dot{\hat{u}} &=& \hat{f}_c (x_p,x_c,\hat y, \hat u) & t \in [t_{j-1},t_j] \\
 \hat{y} (t^+_j) &=& y (t_j)  \\
 \hat{u} (t^+_j) &=& u (t_j)
\end{array}
\end{align}
where $\hat{y} \in \R^{n_y}$ and $\hat{u} \in \R^{n_u}$ are, respectively, the vectors of most recently transmitted plant and controller output values. These two variables are generated by the holding function $\hat f_p$ and $\hat f_c$ between two successive transmission instants. The use of zero-order-hold devices leads to $\hat f_p = 0$ and $\hat f_c = 0$ for instance. In addition, $e := (e_y,e_u) \in \mathbb{R}^{n_e}$ denotes the sampling-induced errors where $e_y := \hat{y} - y\in\Rn[n_y]$ and $e_u := \hat{u} - u\in\Rn[n_u]$. Given $x := (x_p,x_c) \in \R^{n_x}$, it is more convenient to transform \eqref{eq:e09s} into a hybrid system as
\begin{align}
& \left.
\begin{array}{rcl}
\dot{x} &=& f (x,e,w) \\
\dot{e} &=& g (x,e,w) \\
\dot\tau &=& 1 
\end{array} \right\}
\tau \in [0,\tau_{\mathrm{MASP}}] \label{eq:e15s} \\
& \!\!\!\!\left. \begin{array}{rcl}
x^+ &=& x \\
e^+ &=& 0 \\
\tau^+ &=& 0
\end{array} \right\}
\tau \in [\epsilon,\tau_{\mathrm{MASP}}] \label{eq:e16s}
\end{align}
where $\tau \in \Rp$ represents a clock and $w$ denotes the disturbance input. We also have the flow set $\C := \{ (x,e,\tau,w) \colon \tau \in [0,\tau_{\mathrm{MASP}}]\}$ and the jump set $\D := \{ (x,e,\tau,w) \colon \tau \in [\epsilon,\tau_{\mathrm{MASP}}]\}$.

To present our results, we need to make the following assumption.
\begin{assumption} \label{A:01s}
There exist locally Lipschitz functions $V \colon \Rn[n_x] \to \Rp$, $W \colon \mathbb{R}^{n_e} \to \Rp$, a continuous function $H \colon \Rn[n_x] \to \Rp$, $\underline\alpha_x,\overline\alpha_x,\underline\alpha_e,\overline\alpha_e \in \Kinf$, $\tilde\alpha \in \mathcal{PD}$, $\sigma_1,\sigma_2 \in \K$ and real numbers $L,\gamma > 0$ such that the following hold
\begin{align}
& \underline\alpha_x (\abs{x}) \leq V (x) \leq \overline\alpha_x (\abs{x}) \qquad \forall x \in \mathbb{R}^{n_x} , & \label{eq:e17s} 
\end{align}
for all almost $x \in \Rn[n_x]$, for all $e \in \Rn[n_e]$ and all $w \in \Rn[d]$
\begin{align}
& \langle \nabla V (x) , f(x,e,w) \rangle \leq - \tilde\alpha (\abs{x}) - \tilde\alpha (W(e)) - [H(x)]^2 + \gamma^2 [W(e)]^2 + \sigma_1 (\abs{w}) \label{eq:e18s}
\end{align}
moreover,
\begin{align}
& \underline\alpha_e (\abs{e}) \leq W (e) \leq \overline\alpha_e (\abs{e}) \qquad \forall e \in \R^{n_e} & \label{eq:e19s}
\end{align}
and for almost all $e \in \Rn[n_e]$, for all $x \in \R^{n_x}$ and all $w \in \mathbb{R}^d$
\begin{align}
& \left\langle \frac{\partial W (e)}{\partial e} , g(x,e,w) \right\rangle \leq L W (e) + H(x) + \sigma_2 (\abs{w}) . & \label{eq:e21s}
\end{align}
\hfill$\Box$
\end{assumption}
According to \eqref{eq:e17s} and \eqref{eq:e18s}, the emulated controller guarantees the iISS property for subsystem $\dot x = f(x,e,w)$ with $W$ and $w$ as inputs. These properties can be verified by analysis of robustness of the closed-loop system \eqref{eq:e01s}-\eqref{eq:e02s} with respect to input and/or output measurement errors in the absence of digital network. Finally, sufficient conditions under which \eqref{eq:e21s} holds are the function $g$ is globally Lipschitz and there exists $M > 0$ such that $\abs{\frac{\partial W(\kappa,e)}{\partial e}} \leq M$.

The last condition is on the MASP. As in \cite{Nesic.2009}, we need to have a system which has a sufficiently high bandwidth so that the following assumption holds.
\begin{assumption} \label{A:02s}
Let $\tau_{\mathrm{MASP}}$ satisfies $\tau_{\mathrm{MASP}} < \mathcal{T}(\gamma,L)$ where
\begin{align} \label{eq:e22s}
\mathcal{T}(\gamma,L) := \left\{ \begin{array}{ll}
\frac{1}{L r}\tan^{-1}(r) & \gamma>L \\
\frac{1}{L} & L = \gamma \\
\frac{1}{L r}\tanh^{-1}(r) & \gamma<L
\end{array}\right. 
\end{align}
with $r := \sqrt{\abs{(\gamma / L)^2-1}}$. 
\hfill$\Box$
\end{assumption}
Now we are ready to give the main result of this section.
\begin{theorem} \label{T:2}
Let Assumptions~\ref{A:01s} and \ref{A:02s} hold. Then hybrid system~(\ref{eq:e15s}) and~(\ref{eq:e16s}) is iISS with respect to the compact set $\mathcal{A} := \{ (x,e,\tau) \colon x = 0 , e = 0 \}$.
\end{theorem}
\begin{proof}
To prove the theorem, we appeal to Theorem~\ref{T:1}. In particular, we establish hybrid system~\eqref{eq:e15s} and~\eqref{eq:e16s} is smoothly dissipative. On the other hand, hybrid system~\eqref{eq:e15s} and~\eqref{eq:e16s} is also 0-input AS under Assumptions~\ref{A:01s} and \ref{A:02s}, as shown in \cite{Nesic.2009}. Hence, by the implication $(\ref{item:0-gas}) \Rightarrow (\ref{item:iiss})$ of Theorem~\ref{T:1}, (\ref{eq:e23s}) is iISS. Toward the dissipative property of~\eqref{eq:e15s} and~\eqref{eq:e16s}, the following two lemmas are required to give the proof.
\begin{lemma}\label{L:03}
Given $c > 1$ and $\lambda \in (0,1)$, define
\begin{align*} 
\tilde{\mathcal{T}} (c,\lambda,L,\gamma) := \left\{ \begin{array}{ll}
\frac{1}{L r}\tan^{-1}(\frac{r(1-\lambda)}{2(\frac{\lambda}{\lambda+1}) (\frac{\gamma}{L}(\frac{c+1}{2})-1) +1+\lambda}) & L < \gamma\sqrt{c} \\
\frac{1}{L} ( \frac{1-\lambda^2}{\lambda^2 + \frac{\gamma}{L}(1+c)\lambda+1}) & L = \gamma \sqrt{c} \\
\frac{1}{L r}\tanh^{-1}(\frac{r(1-\lambda)}{2(\frac{\lambda}{\lambda+1}) (\frac{\gamma}{L}(\frac{c+1}{2})-1) +1+\lambda}) & L > \gamma \sqrt{c}
\end{array}\right. 
\end{align*}
where $r := \sqrt{\abs{(\gamma / L)^2-c}}$. Let $\phi \colon [0,\tilde{\mathcal{T}}] \to \mathbb{R}$ be the solution to 
\begin{align} \label{eq:e25s}
\dot \phi = -2 L \phi - \gamma (\phi^2+c) \qquad \phi(0) = \lambda^{-1}. 
\end{align}
Then $\phi (\tau) \in [\lambda,\lambda^{-1}]$ for all $\tau \in [0,\tilde{\mathcal{T}}]$.
\end{lemma}
\begin{lemma} \label{L:02}
For any fixed $\gamma$ and $L$, $\tilde{\mathcal{T}}(\cdot,\cdot,\gamma,L) : (1,+\infty) \times (0,1) \to \Rsp$ is continuous and strictly decreasing to zero with respect to the first two arguments.
\end{lemma}
Let $\tau_{\mathrm{MASP}} < \mathcal{T} (\gamma,L)$ be given. For the sake of convenience, denote $\xi := [x^\top,e^\top,\tau]^\top$, $F(\xi,w) := [f(x,e,w)^\top,g(x,e,w)^\top,1]^\top$ and $G(\xi,w) := [x^\top,0^\top,0]^\top$. Also, rewrite hybrid system~\eqref{eq:e15s} and~\eqref{eq:e16s} as
\begin{align} \label{eq:e23s}
& \mathcal{H} := \left\{ \begin{array}{l}
 \dot{\xi} = F (\xi,w) \;\; \quad (\xi,w) \in \mathcal{C} \\
 \xi^{+} = G (\xi,w) \quad (\xi,w) \in \mathcal{D}
 \end{array} \right. . &
\end{align}
It follows from Lemma~\ref{L:02} that there exist $c > 1$ and $\lambda \in (0,1)$ such that $\tau_{\mathrm{MASP}} = \tilde{\mathcal{T}} (c,\lambda,\gamma,L)$. Let the quadruple $(c,\lambda,\gamma,L)$ generate $\phi$ via Lemma~\ref{L:03}. Also, let
\begin{align*}
U (\xi) := V(x) + \gamma \phi(\tau) [W(e)]^2 .
\end{align*}
By (\ref{eq:e17s}), (\ref{eq:e19s}) and the fact that $\phi(\tau) \in [\lambda,\lambda^{-1}]$ for all $\tau \in [0,\tau_{\mathrm{MASP}}]$ (cf. Lemma~\ref{L:03}), there exist $\underline\alpha,\overline\alpha \in \mathcal{K}_\infty$ such that the following holds
\begin{align} \label{eq:e24s}
& \underline\alpha (\abs{[x,e]}) \leq  U (\xi) \leq \overline\alpha (\abs{[x,e]}) . &
\end{align}
For any $(\xi,w) \in \mathcal{C}$, we have
\begin{align*}
& \left\langle \nabla U (\xi),F(\xi,w)\right\rangle = \left\langle V(x), f(x,e,w) \right\rangle + 2 \gamma \phi(\tau) W(e) \left\langle \frac{\partial W}{\partial e} , g(x,e,w) \right\rangle + \gamma \dot\phi(\tau) [W(e)]^2 . & 
\end{align*}
It follows from (\ref{eq:e18s}), (\ref{eq:e21s}) and (\ref{eq:e25s}) that
\begin{align*}
\left\langle \nabla U (\xi),F(\xi,w)\right\rangle \leq & - \tilde\alpha (\abs{x}) - \tilde\alpha (W(e)) - [H(x)]^2 + \gamma^2 [W(e)]^2 + \sigma_1 (\abs{w}) & \nonumber\\
& + 2 \gamma \phi(\tau) W(e) [L W (e) \!+\! H(x) \!+\! \sigma_2 (\abs{w})] - \gamma [2 L \phi \!+ \!\gamma (\phi^2+c)] [W(e)]^2 & \nonumber\\
= & - \tilde\alpha (\abs{x}) - \tilde\alpha (W(e)) - [\gamma \phi(\tau) W(e) - H(x)]^2 - (c-1)\gamma^2 [W (e)]^2 &\nonumber\\
& + \sigma_1 (\abs{w}) + 2 \gamma \phi(\tau) W(e)\sigma_2 (\abs{w}) & \nonumber \\
\leq & - \tilde\alpha (\abs{x}) -\! \tilde\alpha (W(e))\! -\! (c-1)\gamma^2 [W (e)]^2 \!+\! \sigma_1 (\abs{w}) \!+\! 2 \gamma \phi(\tau) W(e)\sigma_2 (\abs{w}) . & 
\end{align*}
From Young's inequality, for any $\varepsilon > 0$ we have 
\begin{align*}
\left\langle \nabla U (\xi),F(\xi,w)\right\rangle \leq & - \tilde\alpha (\abs{x}) - \tilde\alpha (W(e)) - (c-1)\gamma^2 [W (e)]^2 + \sigma_1 (\abs{w}) + \varepsilon \gamma^2 [W(e)]^2 &\nonumber\\
& + \frac{[\phi(\tau)]^2}{\varepsilon} [\sigma_2 (\abs{w})]^2 . & 
\end{align*}
It follows from Lemma~\ref{L:03} that
\begin{align*}
\left\langle \nabla U (\xi),F(\xi,w) \right\rangle \leq & - \tilde\alpha (\abs{x}) - \tilde\alpha (W(e)) - (c-\varepsilon-1) \gamma^2 [W (e)]^2 + \sigma_1 (\abs{w}) \\
&+ \frac{1}{\lambda^2\varepsilon} [\sigma_2 (\abs{w})]^2 . & 
\end{align*}
Given $\sigma (\cdot) := \sigma_1 (\cdot) + \frac{1}{\lambda^2\varepsilon} [\sigma_2 (\cdot)]^2$, we get
\begin{align*}
\left\langle \nabla U (\xi),F(\xi,w) \right\rangle \leq & - \tilde\alpha (\abs{x}) - \tilde\alpha (W(e)) - (c-\varepsilon-1) \gamma^2 [W (e)]^2 + \sigma (\abs{w}) & 
\end{align*}
Picking $\varepsilon$ sufficiently small such that $c-\varepsilon-1 > 0$ gives
\begin{align*} 
\left\langle \nabla U (\xi),F(\xi,w) \right\rangle \leq & - \tilde\alpha (\abs{x}) - \tilde\alpha (W(e)) + \sigma (\abs{w}) . &
\end{align*}
Then 
\begin{align}  \label{eq:e29s}
& \left\langle \nabla U (\xi),F(\xi,w) \right\rangle \leq \sigma (\abs{w}) . &
\end{align}
Also, for any $(\xi,w) \in \mathcal{D}$, we have
\begin{align*}
& U (\xi^+) = V(x^+) + \gamma \phi(\tau^+) [W(e^+)]^2 .& 
\end{align*}
It follows from (\ref{eq:e23s}) that
\begin{align*}
U (\xi^+) = V(x) + \gamma \phi(0) [W(0)]^2 . & 
\end{align*}
By the fact that $W(0) = 0$, we get
\begin{align*}
& U (\xi^+) \leq V(x) \leq U(\xi). & 
\end{align*}
Thus
\begin{align} \label{eq:e27s}
U (\xi^+) - U(\xi) \leq 0. &
\end{align}
for all $(\xi,w) \in \mathcal{D}$. Given (\ref{eq:e24s}), (\ref{eq:e29s}) and (\ref{eq:e27s}), we conclude that (\ref{eq:e23s}) is smoothly dissipative with $\rho(\xi) \equiv 0$ as in~\eqref{eq:e04} and~\eqref{eq:e05}.
\end{proof}
\begin{remark}
Variants of Theorem~\ref{T:2} including a (semiglobal) practical iISS property can be obtained by appropriate modifications to Assumption~\ref{A:01s}. Moreover, motivated by the connections between other engineering systems such as networked control systems and event-triggered control systems with sampled-data systems, we foresee that the application of our results to sampled-data systems can be useful for the study of the iISS property for such hybrid systems.
\hfill$\Box$
\end{remark}
To verify the effectiveness of Theorem~\ref{T:2}, we give an illustrative example. Consider the continuous-time plant with a bounded-input controller
\begin{align}
& \dot x = \sin(x) + u + w & \nonumber \\
& u = - \frac{x}{1+x^2} - \sin(x)  &  \nonumber
\end{align}
where $x,u,w \in \R$. Ignoring the digital channel, the closed-loop system is \emph{not} ISS but iISS. Given the digital communication effects, we write the system into a hybrid system the same as~\eqref{eq:e15s} and~\eqref{eq:e16s}
\begin{eqnarray}\nonumber
&\begin{array}{rcll}
\dot{x} &=& - \frac{x+e_x}{1+(x+e_x)^2} + \sin (x) - \sin(x+e_x) + e_u + w & \quad t \in [t_{j-1},t_j]  \\
\dot{e}_u &=& 0 & \quad t \in [t_{j-1},t_j] \\
\dot{e}_x &=& \frac{x+e_x}{1+(x+e_x)^2} + \sin(x+e_x) - e_u - w & \quad t \in [t_{j-1},t_j] \\
e_u (t^+_j) &=& 0 \\
e_x (t^+_j) &=& 0 . 
\end{array}
\end{eqnarray}
Taking $V(x) = \abs{x},W(e)=\abs{e}$, we have that the requirements in Assumption~\ref{A:01s} are satisfied with $L = 3, \gamma = 10$ and $H(x) = \frac{\abs{x}}{1+x^2}$. The choice of parameters gives $\tau_{\mathrm{MASP}} \simeq 0.13$.

\section{Conclusions} \label{sec:conclusions}
This paper was primarily concerned with Lyapunov characterizations of pre-iISS for hybrid systems. In particular, we established that the existence of a smooth iISS-Lyapunov function is equivalent to pre-iISS which unified and extended results in \cite{Angeli.1999,Angeli.2000}. We also related pre-iISS to dissipativity and detectability notions. Robustness of pre-iISS to vanishing perturbations was investigated, as well. We finally illustrated the effectiveness of our results by providing a maximum allowable sampling period guaranteeing iISS for sampled-data control systems.

Our results can be extended in several directions. In particular, further potential equivalent characterizations of pre-iISS in terms of time-domain behaviors including 0-input pre-AS plus uniform-bounded-energy-bounded-state as well as bounded energy weakly converging state plus 0-input pre-local stability (cf. \cite{Angeli.2004,Angeli.2000b} for the existing equivalent characterizations for continuous-time systems). Moreover, other related notions such as strong iISS, integral input-output-to-state stability and integral output-to-state stability could be investigated.

\section*{Acknowledgments}
	The first author is very grateful to Andy Teel for numerous discussions and for raising the questions which led to this paper. The first author also warmly thanks Dragan Ne\v{s}i\'{c} for his illuminating suggestions and insightful discussions. Particularly, the proof of the implication $(\ref{item:lyapunov-iiss}) \Rightarrow (\ref{item:iiss})$ of Theorem~\ref{T:1} is the fruit of a collaboration with him.

\appendices
\section{A Comparison Lemma for Hybrid Systems} \label{Ap:A}
A generalization of \cite[Lemma C.1]{Cai.2009}, that is a comparison lemma for hybrid systems, to the case of \emph{positive definite} functions is provided by Lemma~\ref{L:7}. Before presenting the lemma, we need to give the following technical results.
\begin{lemma} \label{L:8}
Let $\alpha$ be any positive definite function. Also, let $\rho_1 \in \Kinf$ and $\rho_2 \in \mathcal{L}$ come from Lemma~\ref{L:1} such that $\alpha (r) \geq \rho_1 (r) \rho_2 (r)$ for all $r \geq 0$. Given any absolutely continuous function $w \colon [t_0,\tilde{t}) \to \Rp$ with $\tilde t > t_0$ and $t_0 \geq 0$ satisfying for almost all $t$
\begin{equation} \label{eq:e62}
\dot{w} ( t ) \leq - \alpha (w(t)), \quad w(t_0) \geq 0 ,
\end{equation}
there exists $\beta \in \KL$ such that for all $t \in [t_0,\tilde{t})$ the following holds
\begin{eqnarray}
w (t) \leq \beta ( w ( t_0 ) , \rho_{2} (w(t_0)) (t - t_0) )  \label{eq:e63}
\end{eqnarray}
where $\beta (r , \tau )$ with $\beta (r,0) = r$ is the maximal solution of the differential equation $\frac {\mathrm{d} w } {\mathrm{d} \tau} = - \rho_1 (w)$.
\end{lemma}
\begin{proof} 
Consider any arbitrary $w (t_0) > 0$. Recalling Lemma~\ref{L:1} we get
\begin{align*}
\dot{w} (t) \leq - \rho_1 ( w(t) ) \rho_{2} ( w(t)) .
\end{align*}
Dividing both sides by $\rho_2 ( w (t_0))$ gives
\begin{align*}
\frac{1}{\rho_2 ( w (t_0))} \dot{w} (t) \leq - \frac{ \rho_1 ( w(t) ) \rho_2 ( w(t)) } { \rho_2 ( w (t_0) ) } .
\end{align*}
Define $\tau := \rho_{2} (w (t_0)) t $, and so
\begin{align} \label{eq:e655}
\frac{\mathrm{d} w(\frac{\tau}{\rho_{2} ( w (t_{0}))})}{\mathrm{d} \tau} \leq - \frac { \rho_{1} ( w (\frac{\tau}{\rho_{2} ( w (t_{0}))}) ) \rho_{2} ( w(\frac{\tau}{\rho_{2} ( w (t_{0}))}))} { \rho_{2} ( w (t_{0})) } .
\end{align}
Since $\rho_{2} \in \mathcal{L}$ and \eqref{eq:e62} (i.e. $w$ is non-increasing), we have
\begin{equation} \label{eq:e66}
\frac {\rho_{2} ( w(t))} {\rho_{2} ( w(t_{0})) } \geq 1 .
\end{equation}
Combining (\ref{eq:e655}) with (\ref{eq:e66}) yields
\begin{align} \label{eq:e67}
\frac{\mathrm{d} w}{\mathrm{d} \tau} \leq - \rho_1 (w) .
\end{align}
Without loss of generality, assume that $\rho_1$ is (locally) Lipschitz. Given that (\ref{eq:e67}) and using a standard comparison lemma (e.g. \cite[ Lemma 4.4]{Lin.1996}) give~\eqref{eq:e63}.
\end{proof}
The following lemma is a special version of~\cite[Proposition 1]{Nesic.2004a}.
\begin{lemma} \label{L:9}
Let $\alpha$ be a positive definite function with $\alpha(r) < r$ for all $r > 0$. Also, let $\rho_1 \in \Kinf$ and $\rho_2 \in \mathcal{L}$ come from Lemma~\ref{L:1} such that $\alpha (r) \geq \rho_1 (r) \rho_2 (r)$ for all $r \geq 0$. Given any function $w \colon \Zp \to \Rp$ satisfying for all \textcolor{black}{$j \geq j_0$ with $j_0 \geq 0$} 
\begin{equation} \label{eq:e64}
w (j + 1) - w (j) \leq - \alpha (w(j)), \quad w(j_0) \geq 0 ,
\end{equation}
there exists $\beta \in \KL$ such that for all $j \geq j_0$ the following holds
\begin{eqnarray}
w (j) \leq \beta ( w ( j_0 ) , \rho_{2} (w(j_0)) (j - j_0) )  \label{eq:e65}
\end{eqnarray}
where $\beta (r , \tau)$ with $\beta (r,0) = r$ is the maximal solution of the differential equation $\frac {\mathrm{d} y } {\mathrm{d} \tau} = - \rho_{1} (y)$.
$\textrm{ }$\hfill$\Box$
\end{lemma}

The semigroup property of $\beta (\cdot,\cdot)$ appearing in Lemma~\ref{L:8} and Lemma~\ref{L:9} will be useful in the sequel.
\begin{lemma} (\cite{Sontag.1998b}) \label{L:0}
Let $\beta \in \mathcal{KL}$ come from either Lemma~\ref{L:8} or Lemma~\ref{L:9}. Then the following hold
\begin{itemize}
\item for all $r,s,s_1 \geq 0$ with $s \geq s_1$
\begin{equation} \label{eq:e83}
\beta (\beta (r,s_1) , s - s_1 ) = \beta (r,s) ;
\end{equation}
\item for all $r,s_1,s_2 \geq 0$
\begin{equation} \label{eq:e82}
\beta (\beta (r,s_1),s_2) = \beta (r, s_1 + s_2) = \beta (\beta (r,s_2),s_1) .
\end{equation}
\end{itemize}
\hfill$\Box$
\end{lemma}
\begin{lemma} \label{L:7}
Let $\alpha$ be a positive definite function with $\alpha(r) < r$ for all $r > 0$ satisfying Lemma~\ref{L:1} (i.e. $\alpha (r) \geq \rho_1 (r) \rho_2 (r)$ for all $r \geq 0$). Also consider a hybrid arc $w \colon \dom \, w \to \Rp$ satisfying
\begin{itemize}
\item for almost all $t$ such that $(t,j) \in \dom w \backslash \Gamma(w)$
\begin{equation} \label{eq:e59}
\dot{w}(t,j) \leq - \alpha (w(t,j)) ;
\end{equation} 
\item for all $(t,j) \in \Gamma(w)$ it holds that
\begin{equation} \label{eq:e60}
w(t,j+1) - w(t,j) \leq - \alpha (w(t,j)) .
\end{equation}
\end{itemize}
Then there exists $\tilde{\beta} \in \mathcal{KLL}$ such that for all $(t,j) \in \mathrm{dom} \, w$ the following holds
\begin{align} \label{eq:e86}
& w (t,j) \leq \tilde{\beta} (w(0,0),t,j) .
\end{align} 
Moreover, the function $\tilde\beta$ satisfies the following properties
\begin{align}
& \tilde{\beta} (r,0,0) = r  \qquad \forall r \in \mathbb{R}_{\geq 0} & \label{eq:e24} \\
& \tilde{\beta} (r,s_{1},s_{2}) = \tilde{\beta} (r,s_{2},s_{1}) \qquad \forall r,s_1,s_2 \in \mathbb{R}_{\geq 0}& \label{eq:e25}\\
& \tilde{\beta} (\tilde{\beta} (r,\underline{t},\underline{j}),\bar{t} - \underline{t}, \bar{j} - \underline{j}) = \tilde{\beta} (r,\bar{t},\bar{j}) \quad \forall r,\bar{t},\bar{j},\underline{t},\underline{j} \in \mathbb{R}_{\geq 0} \,\,with\,\, \bar{t} \geq \underline{t}, \bar{j} \geq \underline{j}. & \label{eq:e29}
\end{align}
\end{lemma}
\begin{proof} 
Partition $\mathrm{dom} \, w := \bigcup\limits_{k = 0}^{J} {\left( {\left[ {{t_k},{t_{k + 1}}} \right],k} \right)}$ with $t_0 = 0, t = t_{J+1}$ and $w_{0} := w(0,0)$. Define the function $\tilde{\beta} (w_0, t , j) \in \mathcal{KLL}$ by
$$\tilde{\beta} (w_0, t , j) := \beta (\beta(w_0,\rho_2 (w_0)t),\rho_2 (w_0)j).$$ 
Note that, from (\ref{eq:e82}), $\beta (\beta(w_0,\rho_2 (w_0)t),\rho_2 (w_0)j) = \beta (w_0,\rho_2 (w_0)(t+j))$. Without loss of generality, we assume that $t_{0} \neq t_{1}$. Then $w$ flows for any $t$ in the first flow interval $[0,t_{1}]$. From Lemma~\ref{L:8}, we get
\begin{align} \label{eq:e78}
& w(t,0) \leq \beta (w_0,\rho_2 (w_0) t) = \beta (\beta (w_0,0),\rho_2 (w_0) t) = \tilde{\beta} (w_0, t , 0) &
\end{align}
for all $t \in [0 , t_1]$ with $(t,0) \in \mathrm{dom}\,w$.

Let $J_1$ jumps occur consecutively. It follows from Lemma~\ref{L:9} that for any $(t_k,k) \in \mathrm{dom} \, w$ with $k \in \{ 1, \dots , J_1 \}$
\begin{align*}
& w(t_{k},k) \leq  \beta (w (t_{1},0),\rho_2 (w (t_{1},0)) k) .
\end{align*}
It follows from (\ref{eq:e78}) that
\begin{align*}
& w(t_{k},k) \leq \beta (\beta (w_{0},\rho_2 (w_0) t_1),\rho_{2} (w (t_1,0)) k) .
\end{align*}
By the fact that $\rho_2 (w (t_1,0)) \geq \rho_2 (w_0)$, we have
\begin{align*}
& w(t_k,k) \leq \beta (\beta (w_0,\rho_2 (w_0) t_1),\rho_2 (w_0) k) = \tilde{\beta} (w_0, t_1 , k) .
\end{align*}
It follows from the fact that $t_{1} = \dots = t_{J_{1}}$ that 
\begin{align} \label{eq:e21}
& w(t_k,k) \leq \tilde{\beta} (w_0,t_k,k)
\end{align}
for all $(t_k,k) \in \mathrm{dom}\,w$ with $k \in \{ 1, \dots , J_1 \}$. Let $w(t_{J_1},J_1) \neq 0$, otherwise according to $(\ref{eq:e59})$ and $(\ref{eq:e60})$, $w(t,k) \equiv 0$ for all $(t,j) \in \dom w$ such that $(t,j) \geq (t_{J_1},J_1)$; so the proof is complete. Assume that $t_{J_1} \neq t_{J_1+1}$, so $w(t,J_1)$ flows for all $t$ in the second flow interval of $\dom w$. Then using Lemma~\ref{L:8}, we get
\begin{align*}
& w(t,J_1) \leq \beta ( w ( t_{J_1},J_1 ),\rho_2 (w ( t_{J_1},J_1 )) (t - t_{J_1}) ) .
\end{align*}
It follows from (\ref{eq:e21}) and the fact that $\tilde\beta (w_0,t_{J_1},J_1) = \beta (w_0,\rho_2 (w_0)(t_{J_1}+J_1))$ that
for all $(t,J_1) \in \dom w$
\begin{align*}
& w(t,J_1) \leq \beta ( \beta (w_0,\rho_2 (w_0)(t_{J_1}+J_1)), \rho_2 (w ( t_{J_1},J_1 )) (t - t_{J_1})  ) .
\end{align*}
It follows with the fact that $\rho_{2} (w (t_{J_1},J_1)) \geq \rho_2 (w_0)$ that
\begin{align*}
& w(t,J_1) \leq \beta ( \beta (w_0,\rho_2 (w_0)(t_{J_1}+J_1)), \rho_2 (w_0) (t - t_{J_1}) ) .
\end{align*}
By application of \eqref{eq:e82}, we get
\begin{align*}
& w(t,J_1) \leq \beta ( w_0, \rho_2 (w_0) ( t + J_1 )).
\end{align*}
By reapplication of (\ref{eq:e82}), we have
\begin{align} \label{eq:e23}
& w(t,J_1) \leq \beta ( \beta(w_0,\rho_2 (w_0) t), \rho_2 (w_0) J_1 ) = \tilde\beta(w_0,t,J_1)
\end{align}
for all $t \in [t_{J_1} , t_{J_1 +1}]$ with $(t,J_1) \in \mathrm{dom}\,w$. 

Now $J_2$ jumps happen in a row. Given that Lemma~\ref{L:9} for any $(t_k,k) \in \dom w$ with $k \in \{ J_1 + 1 , \dots , J_2 + J_1 \}$ gives
\begin{align*}
& w(t_k,k) \leq \beta ( w ( t_{J_1 + 1} , J_1 ),\rho_2 (w ( t_{J_1 + 1} , J_1 )) (k - J_{1}) ) . &
\end{align*}
It follows from (\ref{eq:e23}) and the fact that $\tilde\beta(w_0,t_{J_1+1},J_1) = \beta ( w_0, \rho_2 (w_0) ( t_{J_1 + 1} + J_1 ))$ that 
\begin{align*}
& w(t_k,k) \leq \beta ( \beta ( w_0, \rho_2 (w_0) ( t_{J_1 + 1} + J_1 )), \rho_{2} (w( t_{J_{1} + 1} , J_{1} )) (k - J_{1}) ) .
\end{align*}
From $\rho_2 (w ( t_{J_1 + 1} , J_1 )) \geq \rho_2 (w_0)$, we have
\begin{align*}
& w(t_k,k) \leq \beta ( \beta ( w_0, \rho_2 (w_0) ( t_{J_1 + 1} + J_1 )), \rho_{2} (w_0) (k - J_1) ) .
\end{align*}
By application of (\ref{eq:e82}), we have
\begin{align*}
& w(t_k,k) \leq \beta ( w_0, \rho(w_0) (t_{J_1 + 1} + k ) ).
\end{align*}
By reapplication of (\ref{eq:e82}) and the fact that $ t_{J_1 + 1} = \dots = t_{J_{1} + J_{2}} $, we get
\begin{align*}
& w(t_k,k) \leq \beta ( \beta ( w_0, \rho(w_0) t_{J_1 + 1}) , \rho(w_0) k ) = \tilde\beta(w_0,t_k,k)
\end{align*}
for all $(t_k,k) \in \dom \, w$ with $k \in \{ J_{1} + 1 , \dots , J_2 + J_1 \}$.

By repeated application of the above arguments (i.e. concatenating flows and jumps, the fact that $\rho_2 (w ( t,j )) \geq \rho_2 (w_0)$ and exploiting (\ref{eq:e82})) yield 
\begin{align} \nonumber
& w(t,j) \leq \tilde {\beta} (w_0 , t , j) \qquad\forall (t,j) \in \dom w .
\end{align}
It is easy to see that when $w$ starts with jumps the above arguments essentially hold. Eventually, the properties (\ref{eq:e24})-(\ref{eq:e29}) immediately follow from the very definition of $\tilde\beta$ and exploiting (\ref{eq:e82}). This completes the proof.
\end{proof}
\section{Proof of Lemma~\ref{L:2}} \label{Ap:E}

Partition $\dom z = \bigcup\limits_{k = 0}^{J-1} {\left( {\left[ {{t_k},{t_{k + 1}}} \right],k} \right)} \bigcup ( [t_{J} , t_{J+1}) , J ) $ with $t_{0} = 0$ and $(\tilde{t},\tilde{j}) := (t_{J+1},J)$. Define $(\hat{t},\hat{j})$ by
\begin{align}
(\hat{t},\hat{j}) := \min \big\{ & (t,j) \in \dom \, z \colon   z(t,j) \leq \norm{v_{(t,j)}} _\infty \,\,\mathrm{and} \,\, t+j < \tilde{t}+\tilde{j} \big\} \nonumber
\end{align}
By convention, if $ z(t,j) > \norm{v_{(t,j)}}_{\infty}$ for all $(t,j) \in \dom \, z$, we let $(\hat{t},\hat{j}) := (\tilde{t},\tilde{j})$. \\
Nonincreasing property of $z(t,j)$ immediately follows from~(\ref{eq:e6}) and~(\ref{eq:e7}). Therefore, we get for all $(t,j) \in \mathrm{dom} \, z$ such that $(t,j) \succeq (\hat{t},\hat{j})$
\begin{align} \label{eq:e28}
z(t,j) \leq \norm{v_{(t,j)}}_{\infty} .
\end{align}
Thus (\ref{eq:e8}) holds for all $(t,j) \in \mathrm{dom} \, z$ such that $(t,j) \succeq (\hat{t},\hat{j})$. 

Pick any $(t,j) \in \mathrm{dom} \, z$ such that $\left( {t,j} \right) \prec (\hat{t},\hat{j})$. We have that $z(t,j) > \abs{v_{(t,j)}}_\infty \geq v(\tau ,i)$ for all $(t,j),(\tau,i) \in \mathrm{dom} \, z$ such that $(0,0) \preceq (\tau ,i) \preceq (t,j)$. From the fact that $z(0,0) \geq 0$ and the fact that $z(t,j)$ is non-increasing, we have
\begin{equation} \label{eq:e26}
0 \leq z(\tau,i) \leq z(\tau,i) + v(\tau,i) \leq 2z(\tau,i)
\end{equation}
for all $(\tau,i) \in \mathrm{dom} \, z$ with $(0,0) \preceq (\tau,i) \preceq (t,j)$. \\
Combining (\ref{eq:e6}) and (\ref{eq:e7}) with the fact that $z(t,j) + v(t,j) \geq 0$ for all $(t,j) \prec (\hat{t},\hat{j})$ (cf. the second inequality of (\ref{eq:e26})) gives
\begin{itemize}
\item for almost all $t$ such that $(t,j) \in {\mathrm{dom}} \, z \backslash \Gamma(z)$ and $(t,j) \prec (\hat{t},\hat{j})$
\begin{align}
& \dot{z} (t,j) \leq - \rho (z(t,j)+v(t,j)) ; & \label{eq:e31}
\end{align}
\item for all $(t,j) \in \Gamma(z)$ such that $(t,j) \prec (\hat{t},\hat{j})$
\begin{align}
& z(t,j+1) - z(t,j) \leq  - \rho (z(t,j)+v(t,j)) . & \label{eq:e32}
\end{align}
\end{itemize}
Applying Lemma~\ref{L:1} to the right-hand sides of~(\ref{eq:e31}) and~(\ref{eq:e32}) (i.e. $\rho (\cdot) \geq \rho_1(\cdot) \rho_2(\cdot)$ for some $\rho_1 \in \mathcal{K}_\infty$ and $\rho_2 \in \mathcal{L}$) gives
\begin{itemize}
\item for almost all $t$ such that $(t,j) \in {\mathrm{dom}} \, z \backslash \Gamma(z)$ and $(t,j) \prec (\hat{t},\hat{j})$
\begin{align*}
& \dot{z} (t,j) \leq - \rho_1 (z(t,j)+v(t,j)) \rho_2 (z(t,j)+v(t,j)) ; & \nonumber
\end{align*}
\item for all $(t,j) \in \Gamma(z)$ such that $(t,j) \prec (\hat{t},\hat{j})$
\begin{align*}
& z(t,j+1) - z(t,j) \leq  - \rho_1 (z(t,j)+v(t,j)) \rho_2 (z(t,j)+v(t,j)) . & \nonumber 
\end{align*}
\end{itemize}
Exploiting the inequalities of (\ref{eq:e26}) and the monotonicity of $\rho_1$ and $\rho_2$ yields
\begin{itemize}
\item for almost all $t$ such that $(t,j) \in {\mathrm{dom}} \, z \backslash \Gamma(z)$ and $(t,j) \prec (\hat{t},\hat{j})$
\begin{align*}
& \dot{z} (t,j) \leq - \rho_1 \left( {z(t,j)} \right)\rho_2 \left( {2z(t,j)} \right) =: - \alpha(z(t,j)) ; &
\end{align*}
\item for all $(t,j) \in \Gamma(z)$ such that $(t,j) \prec (\hat{t},\hat{j})$
\begin{align*}
& z(t,j+1) - z(t,j) \leq - {\rho _1}\left( {z(t,j)} \right){\rho _2}\left( {2z(t,j)} \right) = - \alpha(z(t,j)) . &
\end{align*}
\end{itemize}
By application of Lemma \ref{L:7}, there exists $\tilde{\beta} \in \mathcal{KLL}$ such that
\begin{align} \label{eq:e27}
z(t,j) \leq \tilde{\beta} (z(0,0),t,j) 
\end{align}
for all $(t,j) \in \mathrm{dom} \, z$ with $(t,j) \prec (\hat{t},\hat{j})$. The combination of (\ref{eq:e27}) with (\ref{eq:e28}) completes the proof.
\section{Proof of Theorem~\ref{P:02}} \label{Ap:H}

Before proceeding to the proof, we make the following observation followed by two new notions.
\begin{remark} \label{D:01}
It should be pointed out that there is no loss of generality in working with $\KL$ functions rather than $\KLL$ functions (cf. \cite[Lemma 6.1]{Cai.2007} for more details).
Moreover, we note that $\max \{a,b,c\} \leq a+b+c \leq \max \{3a,3b,3c\}$ for all $a,b,c \in \Rp$.
Hence, $\H$ is pre-iISS with respect to $\A$ if and only if there exist $\alpha \in \Kinf$, $\gamma_1,\gamma_2 \in \K$ and $\beta \in \KL$ if for all $u \in \mathcal{L}_{\gamma_1,\gamma_2}^e$, all $\xi \in \X$, and all $(t,j) \in \dom x$, each solution pair $(x,u)$ to $\H$ satisfies
\begin{align} 
\alpha(\omega(x(t,j,\xi,u))) \leq & \max \Bigg\{ \beta (\omega(\xi),t+j) , \int_0^t \gamma_1 (\abs{u(s,i(s))}) \textmd{d}s , \!\!\!\!\!\!\!\!\!\!\!\!\sum_{\scriptsize{\begin{array}{c}
(t^\prime,j^\prime) \in \Gamma(u),\\
(0,0) \preceq (t^{\prime},j^{\prime}) \prec (t,j)\end{array}}} \!\!\!\!\!\!\!\!\!\!\!\!\gamma_2 (\abs{u(t^\prime,j^\prime)}) 
 \Bigg\} . \label{eq:er1}
\end{align}
\hfill$\Box$
\end{remark}
For the sake of convenience, here we prefer to use the max-type estimate \eqref{eq:er1} rather than \eqref{eq:e1}. The next two notions are required later.
\begin{definition} \label{D:02}
Let $\A \subset \X$ be a compact set, and $\sigma \colon \X \to \Rp$ be an admissible perturbation radius that is positive on $\X \backslash \A$.
Also, let $\omega$ be a proper indicator for $\A$ on $\X$.
The hybrid system $\H$ is said to be semiglobally practically robustly pre-integral input-to-state stable $($SPR-pre-iISS$)$ with respect to $\A$ if there exist $\alpha \in \Kinf$, $\beta \in \KL$, $\gamma_1,\gamma_2 \in \K$ such that for each pair of positive real numbers $(\varepsilon,r)$,
there exists $\delta^* \in (0,1)$ such that for any $\delta \in (0,\delta^*]$ each solution pair $(\overline{x},u)$ to $\H_{\delta \sigma}$, the $\delta \sigma$-perturbation of $\mathcal{H}$, exists for all $u \in \mathcal{L}_{\gamma_1,\gamma_2}^e(r)$, all $\overline{\xi} \in {\X}$ with $\omega(\overline{\xi}) \leq r$ and all $(t,j) \in \mathrm{dom}\,\overline{x}$, and also satisfies
\begin{align*}
\alpha(\omega(\overline{x}(t,j,\overline{\xi},u))) \leq & \max \Bigg\{ \beta (\omega(\overline{\xi}),t+j) , \int_0^t \gamma_1 (\abs{u(s,i(s))}) \textmd{d}s , \!\!\!\!\!\!\!\!\!\!\!\! \sum_{\scriptsize{\begin{array}{c}
(t^\prime,j^\prime) \in \Gamma(u),\\
(0,0) \preceq (t^{\prime},j^{\prime}) \prec (t,j)\end{array}}} \!\!\!\!\!\!\!\!\!\!\!\!\gamma_2 (\abs{u(t^\prime,j^\prime)}) 
 \Bigg\} + \varepsilon .
\end{align*}
\hfill$\Box$
\end{definition}
\begin{definition} \label{D:03}
Let $\A \subset {\X}$ be a compact set, and $\sigma \colon {\X} \to \Rp$ be an admissible perturbation radius that is positive on ${\X} \backslash \A$.
Also, let $\omega$ be a proper indicator for $\A$ on ${\X}$.
The hybrid system $\H$ is said to be SPR-pre-iISS with respect to $\A$ on finite time intervals if there exist $\alpha\in \Kinf$, $\gamma_1,\gamma_2 \in \K$ and $\beta \in \KL$ such that for each triple of positive real numbers $(T,\varepsilon,r)$,
there exists $\delta^* \in (0,1)$ such that for any $\delta \in (0,\delta^*]$ each solution pair $(\overline{x},u)$ to $\H_{\delta \sigma}$, the $\delta \sigma$-perturbation of $\H$, exists for all $u \in \mathcal{L}_{\gamma_1,\gamma_2}^e (r)$, all $\overline{\xi} \in {\X}$ with $\omega(\overline{\xi}) \leq r$ and all $(t,j) \in \dom\,\overline{x}$ with $t+j \leq T$, and also satisfies
\begin{align*}
\alpha(\omega(\overline{x}(t,j,\overline{\xi},u))) \leq & \max \Bigg\{ \beta (\omega(\overline{\xi}),t+j) , \int_0^t \gamma_1 (\abs{u(s,i(s))}) \textmd{d}s , \!\!\!\!\!\!\!\!\!\!\!\!  \sum_{\scriptsize{\begin{array}{c}
(t^\prime,j^\prime) \in \Gamma(u),\\
(0,0) \preceq (t^{\prime},j^{\prime}) \prec (t,j)\end{array}}} \!\!\!\!\!\!\!\!\!\!\!\!\gamma_2 (\abs{u(t^\prime,j^\prime)}) 
 \Bigg\} + \varepsilon .
\end{align*}
\hfill$\Box$
\end{definition}
Here are the steps of the proof: 1) we show that semiglobal practical robust pre-iISS on compact time intervals is equivalent to semiglobal practical robust pre-iISS on the semi-infinite interval (cf. Proposition~\ref{P:00} below\footnote{Without loss of generality, in this proposition, we assume that the length of hybrid time domain of interest is infinite.}); 2) we establish that if solutions of some inflated system can be made arbitrarily close on arbitrary compact time intervals to some solution of the original system when the original system is pre-iISS, then the inflated system is semiglobally practically robustly pre-iISS (cf. Proposition~\ref{P:01} below). 3) we show that semiglobal practical robust pre-iISS implies pre-iISS (cf. Proposition~\ref{T:00} below). 4) the combination of Proposition~\ref{P:01} and Proposition~\ref{T:00} provides what we need, that is to say, the existence of an inflated hybrid system remaining pre-iISS under sufficiently small perturbations when the original system is pre-iISS.

The first step provides a link between the last two definitions.
\begin{proposition} \label{P:00}
The following are equivalent
\begin{itemize}
    \item [A)] $\mathcal{H}$ is SPR-pre-iISS with respect to $\A$ on finite time intervals.
\item [B)] $\mathcal{H}$ is SPR-pre-iISS with respect to $\A$.
\end{itemize}
\end{proposition}
\begin{proof} The implication $A) \Rightarrow B)$ is clear. To establish the implication $B) \Rightarrow A)$, let the gain functions $\alpha$, $\beta$, $\gamma_1$ and $\gamma_2$ come from Remark~\ref{D:01}. Take arbitrary strictly positive $\varepsilon,r$, and let $T > 0$ be sufficiently large such that
\begin{align} \label{eq:r7}
\beta \left( \max \{ r , r + \varepsilon , \alpha^{-1} (r + \varepsilon) \} , s \right) \leq \frac{\varepsilon}{2} \qquad\qquad \forall s \in [ T , \infty ) .
\end{align}
Let $\delta^* \in (0,1)$ come from the assumption of SPR-pre-iISS on finite time intervals, corresponding to the values $(2T,\frac{\varepsilon}{2},\max \left\{ r , r + \varepsilon , \alpha^{-1} (r + \varepsilon) \right\})$. Let $\delta$ be fixed but arbitrary with $\delta \in (0,\delta^*]$. So for all $u \in \mathcal{L}_{\gamma_1,\gamma_2}^e (r)$, for all $\overline{\xi} \in {\X}$ with $\omega(\overline{\xi}) \leq \max \left\{ r , r + \varepsilon , \alpha^{-1}(r + \varepsilon) \right\}$ and for all $(t,j) \in \mathrm{dom}\,\overline{x}$ with $t+j \leq 2T$, each solution pair $(\overline{x},u)$ to $\H_{\delta \sigma}$ exists and satisfies
\begin{align}
\alpha(\omega(\overline{x}(t,j,\overline{\xi},u))) \leq & \max \Bigg\{ \beta (\omega(\overline{\xi}),t+j), \int_0^t \gamma_1 (\abs{u(s,i(s))}) \textmd{d}s ,\!\!\!\!\!\!\!\!\!\!\!\! \sum_{\scriptsize{\begin{array}{c}
(t^\prime,j^\prime) \in \Gamma(u),\\
(0,0) \preceq (t^{\prime},j^{\prime}) \prec (t,j)\end{array}}} \!\!\!\!\!\!\!\!\!\!\!\!\!\!\gamma_2 (\abs{u(t^\prime,j^\prime)})  \Bigg\} + \frac{\varepsilon}{2} . & \nonumber
\end{align}
Let $(t_{kT},j_{kT}) := (t,j)$ with $t + j = kT , k =0,1,2,\dots$ and $(t,j) \in \mathrm{dom}\,\overline{x}$. It follows with the fact that $\omega(\overline{\xi}) \leq \max \left\{ r,r + \varepsilon , \alpha^{-1}(r + \varepsilon) \right\}$, the fact that $u \in \mathcal{L}_{\gamma_1,\gamma_2}^e (r)$, and the choice of $T$ (cf. (\ref{eq:r7})) that
\begin{align} 
\alpha(\omega(\overline{x} (t_T,j_T,\overline{\xi},u))) \leq & \max \Bigg\{ \beta (\omega(\overline{\xi}),T), \int_0^{t_T} \gamma_1 (\abs{u(s,i(s))}) \textmd{d}s , \!\!\!\!\!\!\!\!\!\!\!\!\!\!\!\!\sum_{\scriptsize{\begin{array}{c}
(t^\prime,j^\prime) \in \Gamma(u),\\
(0,0) \preceq (t^{\prime},j^{\prime}) \prec (t_T,j_T)\end{array}}} \!\!\!\!\!\!\!\!\!\!\!\!\!\!\!\!\!\!\! \gamma_2 (\abs{u(t^\prime,j^\prime)}) \Bigg\} + \frac{\varepsilon}{2} & \nonumber \\
\leq & \max \Big\{ \beta (\max \{r,r+\varepsilon , \alpha^{-1} (r + \varepsilon) \},T),r \Big\} + \frac{\varepsilon}{2} & \nonumber \\
\leq & r + \varepsilon .  \label{eq:r2}
\end{align}
It follows from (\ref{eq:r2}) that the following holds
\begin{align} \label{eq:r20}
\omega(\overline{x}(t_T,j_T,\overline{\xi},u)) \leq \alpha^{-1} (r + \varepsilon) \leq \max \left\{ r , r + \varepsilon , \alpha^{-1} (r + \varepsilon) \right\}
\end{align}
Exploiting the semigroup property of solutions, (\ref{eq:r20}) and the fact that $u \in \mathcal{L}_{\gamma_1,\gamma_2}^e (r)$, and the choice of $\delta$, the solution pair $(\overline{x},u)$ to $\H_{\delta \sigma}$ with the initial value $\overline{x}(t_T,j_T,\overline{\xi},u)$ exists for all $(t,j) \in \dom \,\overline{x}$ with $T \leq t+j \leq 3T$ and it also satisfies
\begin{align*} 
\alpha(\omega(\overline{x}(t,j,\overline{x}(t_T,j_T,\overline{\xi},u),u))) \leq & \max \Bigg\{ \beta (\omega(\overline{x}(t_T,j_T,\overline{\xi},u)),t+j) , \int_{t_T}^t \gamma_1 (\abs{u(s,i(s))}) \mathrm{d}s , \nonumber \\
& \qquad\quad  \!\!\!\!\!\!\!\!\!\!\!\!\!\!\!\!{\color{black} \sum_{\scriptsize{\begin{array}{c}
(t^\prime,j^\prime) \in \Gamma(u),\\
(t_T,j_T) \preceq (t^{\prime},j^{\prime}) \prec (t,j)\end{array}}} \!\!\!\!\!\!\!\!\!\!\!\!\!\!\!\!\gamma_2 (\abs{u(t^\prime,j^\prime)}) } \Bigg\} + \frac{\varepsilon}{2} .
\end{align*}
Again it follows from (\ref{eq:r20}), the fact that $u \in \mathcal{L}_{\gamma_1,\gamma_2}^e (r)$, and (\ref{eq:r7}) that
\begin{align*}
\alpha(\omega(\overline{x}(t_{2T},j_{2T},\overline{x}(t_{T},j_{T},\overline{\xi},u),u))) \leq & \max \Bigg\{ \beta (\alpha^{-1} (r + \varepsilon),2T), \int_{t_T}^{t_{2T}} \gamma_1 (\abs{u(s,i(s))}) \textmd{d}s, \nonumber \\
&  \quad\qquad \!\!\!\!\!\!\!\!\!\!\!\!\!\!\!\!{\color{black} \sum_{\scriptsize{\begin{array}{c}
(t^\prime,j^\prime) \in \Gamma(u),\\
(t_T,j_T) \preceq (t^{\prime},j^{\prime}) \prec (t_{2T},j_{2T})\end{array}}} \!\!\!\!\!\!\!\!\!\!\!\!\!\!\!\!\gamma_2 (\abs{u(t^\prime,j^\prime)}) } \Bigg\} + \frac{\varepsilon}{2} & \nonumber \\
\leq & \max \left\{ \beta (\max \{r,r+\varepsilon , \alpha^{-1} (r + \varepsilon) \},2T),r \right\} + \frac{\varepsilon}{2} & \nonumber \\
\leq & r + \varepsilon . &
\end{align*}
By repeating this procedure, the following holds for all $(t,j) \in \mathrm{dom}\,\overline{x}$ with $kT \leq t+j \leq (k+2)T,k=2,3,4,\dots$
\begin{align*}
\alpha(\omega(\overline{x}(t,j,\overline{x}(t_{kT},j_{kT},\overline{\xi},u),u))) \leq & \max \Bigg\{ \beta (\omega(\overline{x}(t_{kT},j_{kT},\overline{\xi},u)),t+j) , \int_{t_{kT}}^t \gamma_1 (\abs{u(s,i(s))}) \mathrm{d}s , \nonumber \\
& \qquad\quad \!\!\!\!\!\!\!\!\!\!\!\!\!\!\!\!{\color{black} \sum_{\scriptsize{\begin{array}{c}
(t^\prime,j^\prime) \in \Gamma(u),\\
(t_{kT},j_{kT}) \preceq (t^{\prime},j^{\prime}) \prec (t,j)\end{array}}} \!\!\!\!\!\!\!\!\!\!\!\!\!\!\!\!\gamma_2 (\abs{u(t^\prime,j^\prime)}) } \Bigg\} + \frac{\varepsilon}{2} \nonumber \\
\leq & \max \Bigg\{ \int_{t_{kT}}^t \gamma_1 (\abs{u(s,i(s))}) \mathrm{d}s , \!\!\!\!\!\!\!\!\!\!\!\!\!\!\!\!\!\!\!\!\!\!{\color{black} \sum_{\scriptsize{\begin{array}{c}
(t^\prime,j^\prime) \in \Gamma(u),\\
(t_{kT},j_{kT}) \preceq (t^{\prime},j^{\prime}) \prec (t,j)\end{array}}} \!\!\!\!\!\!\!\!\!\!\!\!\!\!\!\!\!\!\!\!\!\gamma_2 (\abs{u(t^\prime,j^\prime)}) } \Bigg\} + \varepsilon . &
\end{align*}
So we have for all $\overline{\xi} \in {\X}$ with $\omega(\overline{\xi}) \leq \max \left\{ r , r + \varepsilon , \alpha^{-1}(r + \varepsilon) \right\}$, for all $u \in \mathcal{L}_{\gamma_1,\gamma_2}^e(r)$, and for all $(t,j) \in \mathrm{dom}\,\overline{x}$ with $t+j \geq T$
\begin{align} \label{eq:r4}
\alpha(\omega(\overline{x}(t,j,\overline{\xi},u))) \leq & \max \Bigg\{ \int_0^t \gamma_1 (\abs{u(s,i(s))}) \mathrm{d}s , \!\!\!\!\!\!\!\!\!\!\!\!\!\!\!\!{ \sum_{\scriptsize{\begin{array}{c}
(t^\prime,j^\prime) \in \Gamma(u),\\
(0,0) \preceq (t^{\prime},j^{\prime}) \prec (t,j)\end{array}}} \!\!\!\!\!\!\!\!\!\!\!\!\!\!\!\!\gamma_2 (\abs{u(t^\prime,j^\prime)})} \Bigg\} + \varepsilon . &
\end{align}
We also get for all $\overline{\xi} \in {\X}$ with $\omega(\overline{\xi}) \leq \max \left\{ r , r + \varepsilon , \alpha^{-1}(r + \varepsilon) \right\}$, for all $u \in \mathcal{L}_{\gamma_1,\gamma_2}^e(r)$, and for all $(t,j) \in \dom \overline{x}$ with $0 \leq t+j < T$
\begin{align}
\alpha(\omega(\overline{x}(t,j,\overline{\xi},u))) \leq & \max \Bigg\{ \beta (\omega(\overline{\xi}),t+j), \int_0^t \gamma_1 (\abs{u(s,i(s))}) \textrm{d}s , \!\!\!\!\!\!\!\!\!\!\!\!\!\!\!\!{\color{black} \sum_{\scriptsize{\begin{array}{c}
(t^\prime,j^\prime) \in \Gamma(u),\\
(0,0) \preceq (t^{\prime},j^{\prime}) \prec (t,j)\end{array}}} \!\!\!\!\!\!\!\!\!\!\!\!\!\!\!\!\gamma_2 (\abs{u(t^\prime,j^\prime)}) } \Bigg\} + \frac{\varepsilon}{2} . & \label{eq:r5}
\end{align}
Combining (\ref{eq:r4}) and (\ref{eq:r5}) gives
\begin{align*}
\alpha(\omega(\overline{x}(t,j,\overline{\xi},u))) \leq & \max \left\{ \beta (\omega(\overline{\xi}),t+j), \norm{u_{(t,j)}}_{\gamma_1,\gamma_2} \right\} + \varepsilon & 
\end{align*}
which completes the proof. 
\end{proof}
The following concepts, borrowed from \cite{Goebel.2012}, are required to give Proposition~\ref{P:01}.
\begin{definition}
Two hybrid signals $x \colon \textrm{dom } x \to \Rn$ and $y \colon \dom y \to \Rn$ are said to be ($T,\varepsilon$)-close if
\begin{enumerate}
    \item for each $(t,j) \in \textrm{dom }x$ with $t + j \leq T$ there exists $s$ such that $(s,j) \in \dom y$, with $\abs{t - s} \leq \varepsilon$ and $\abs{x(t, j) - y(s, j)} \leq \varepsilon$;
\item for each $(t,j) \in \textrm{dom }y$ with $t + j \leq T$ there exists $s$ such that $(s,j) \in \textrm{dom }x$, with $\abs{t - s} \leq \varepsilon$ and $\abs{x(t, j) - y(s, j)} \leq \varepsilon$.
\end{enumerate}
\hfill$\Box$
\end{definition}
\begin{definition} \label{D:04}
(Reachable Sets) Given an arbitrary compact set $K_0 \subset {\X}$ and $T \in \Rp$, the reachable set from $K_0$ in hybrid time less or equal to $T$ is the set
\begin{align*}
\mathcal{R}_{\leq T} (K_0) = \{ x (t,j,\xi,u) \colon x \in \varrho_\H (\xi) , \xi \in K_0 , t+j \leq T \} .
\end{align*}
\hfill$\Box$
\end{definition}
We now give a result stating that if solutions to $\mathcal{H}$ and solutions to $\mathcal{H}_{\delta \sigma}$, the $\delta \sigma$ perturbation of $\mathcal{H}$, are ($T,\varepsilon$)-close when $\mathcal{H}$ is pre-iISS, then the system $\mathcal{H}$ is SPR-pre-iISS.
\begin{proposition} \label{P:01}
Let $\A \subset {\X}$ be a compact set and $\sigma \colon {\X} \to \Rp$ be an admissible perturbation radius that is positive on ${\X} \backslash \A$. Also, let $\omega$ be a proper indicator for $\A$ on ${\X}$. Assume that the following conditions hold
\begin{itemize}
    \item[(\emph{a})] $\H$ is pre-iISS with respect to $\A$.
	 \item[(\emph{b})] For each triple $(T,\tilde{\varepsilon},r)$ of positive real numbers there exists some $\delta \in (0,1)$ such that each solution pair $(\overline{x},u)$ to $\H_{\delta \sigma}$, the $\delta \sigma$-perturbation of $\H$, with $\omega(\overline{\xi}) \leq r + \delta$ and $u \in \mathcal{L}_{\gamma_1,\gamma_2}^e (r)$ there exist a solution pair $(x,w)$ to $\H$ with $\omega(\xi) \leq r$, and $\norm{w_{(s,j)}}_{\gamma_1,\gamma_2} \leq \norm{u_{(t,j)}}_{\gamma_1,\gamma_2}$ for all $\abs{t-s} \leq \tilde{\varepsilon}$, $(s,j) \in \dom \, w$ and $(t,j) \in \dom \, u$ such that $\overline{x}$ and $x$ are $(T,\tilde{\varepsilon})$-close. 
\end{itemize}
Then $\mathcal{H}$ is SPR-pre-iISS with respect to $\A$.
\end{proposition}
\begin{proof} This is proved using steps in the proof of \cite[Proposition 3]{Wang.2012}.
From the result of Proposition~\ref{P:00}, we only need to show that $\H$ is SPR-pre-iISS with respect to $\A$ on finite time intervals.
Assume that $\sigma \colon {\X} \to \Rp$ is an admissible perturbation radius that is positive on $\xi \in {\X} \backslash \A$.
Let $\omega$ be a proper indicator for $\A$ on ${\X}$.
Also, let the functions $\alpha\in\Kinf$, $\beta\in \KL$ and $\gamma_1,\gamma_2 \in \K$ come from Definition~\ref{D:03}.
Let the triple of $(T,\varepsilon,r)$ be given.
Let $K_0 := \left\{ \xi \in {\X} \colon \omega (\xi) \leq r \right\}$.
It is clear that $K_0$ is a compact set.
Let $\mathcal{R}_{\leq T} (K_0)$ be the reachable set from $K_0$ for $\H$.
It follows from \cite[Lemma 6.16]{Goebel.2012} that the set $\mathcal{R}_{\leq T} (K_0)$ is compact because $\H$ is pre-iISS.
Using the continuity of $\omega$ and $\beta$, and the fact that $\beta(s,l) \to 0$ as $l \to +\infty$, let $\tilde{\varepsilon}_1 > 0$ be sufficiently small so that
\begin{align*}
\beta(s,l-\tilde{\varepsilon}_1) - \beta(s,l) \leq \frac{\varepsilon}{4} \qquad \forall s \leq r , l \geq 0 .
\end{align*}
By convention, $l = l - \tilde{\varepsilon}_{1}$ if $l - \tilde{\varepsilon}_{1} < 0$.

Let $\tilde{\varepsilon}_{2}$ be small enough such that for all $x \in \mathcal{R}_{\leq T} (K_0)$ and $\overline{x} \in \mathcal{R}_{\leq T} (K_0 + \tilde{\varepsilon}_2 \overline{\mathbb{B}})$ satisfying $\abs{x - \overline{x}} \leq \tilde{\varepsilon}_2$ we have
\begin{align*}
& \alpha(\omega(\overline{x})) \leq \alpha(\omega(x)) + \frac{\varepsilon}{4} , \\
& \beta(\omega(x),l) \leq \beta(\omega(\overline{x}),l) + \frac{\varepsilon}{4} \qquad l \geq 0 . 
\end{align*}
Let $\tilde{\varepsilon} := \min \{ \tilde{\varepsilon}_1 , \tilde{\varepsilon}_2 \}$.
Let the data $(T,\tilde{\varepsilon},r)$ generate $\delta > 0$ from the item (b) of Proposition~\ref{P:01}.
From this item, for each solution pair $(\overline{x},u)$ to $\mathcal{H}_{\delta \sigma}$ with $\overline{\xi} \in (K_0 + \delta \overline{\mathbb{B}})$ and $u \in \mathcal{L}_{\gamma_{1},\gamma_{2}}^e (r)$ there exists some solution pair $(x,w)$ to $\mathcal{H}$ with $\xi \in K_{0}$ and $\norm{w_{(s,j)}}_{\gamma_{1},\gamma_{2}} \leq \norm{u_{(t,j)}}_{\gamma_{1},\gamma_{2}}$ for all $\abs{t-s} \leq \tilde{\varepsilon}$, $(s,j) \in \textrm{dom }w$ and $(t,j) \in \textrm{dom }u$ such that $\overline{x}$ and $x$ are $(T,\tilde{\varepsilon})$-close.
It follows from the item (a) of Proposition~\ref{P:01} and the definition of $\tilde{\varepsilon}$ that for all $(t,j) \in \mathrm{dom}\,\overline{x}$ with $t+j \leq T$, each solution pair $(\overline{x},u)$ to $\mathcal{H}_{\delta \sigma}$ with $\overline{\xi} \in (K_{0}+\delta \overline{\mathbb{B}})$ and $u \in \mathcal{L}_{\gamma_{1},\gamma_2}^e (r)$ satisfies
\begin{align*}
\alpha(\omega(\overline{x}(t,j,\overline{\xi},u))) & \leq \alpha(\omega(x(s,j,\xi,w))) + \frac{\varepsilon}{4} & \nonumber \\
& \leq \max \Bigg\{ \beta (\omega(\xi),t+j-\tilde{\varepsilon}) , \int_0^s \gamma_1 (\abs{w(\tau,i(\tau))}) \mathrm{d}\tau , \!\!\!\!\!\!\!\!\!\!\!\!\!\!\!\!\!\!\!\!{ \sum_{\scriptsize{\begin{array}{c}
(t^\prime,j^\prime) \in \Gamma(w),\\
(0,0) \preceq (t^{\prime},j^{\prime}) \prec (s,j)\end{array}}} \!\!\!\!\!\!\!\!\!\!\!\!\!\!\!\!\!\!\!\!\!\gamma_2 (\abs{w(t^\prime,j^\prime)}) } \Bigg\} + \frac{\varepsilon}{4} & \nonumber \\
& \leq \max \Bigg\{ \beta (\omega(\xi),t+j) , \int_0^s \gamma_1 (\abs{w(\tau,i(\tau))}) \mathrm{d}\tau , \!\!\!\!\!\!\!\!\!\!\!\!\!\!\!\!{ \sum_{\scriptsize{\begin{array}{c}
(t^\prime,j^\prime) \in \Gamma(w),\\
(0,0) \preceq (t^{\prime},j^{\prime}) \prec (s,j)\end{array}}} \!\!\!\!\!\!\!\!\!\!\!\!\!\!\!\!\gamma_2 (\abs{w(t^\prime,j^\prime)}) } \Bigg\} + \frac{\varepsilon}{2} & \nonumber \\
& \leq \max \Bigg\{ \beta (\omega(\overline{\xi}),t+j) , \int_0^t \gamma_1 (\abs{u(\tau,i(\tau))}) \mathrm{d}\tau , \!\!\!\!\!\!\!\!\!\!\!\!\!\!\!\!{ \sum_{\scriptsize{\begin{array}{c}
(t^\prime,j^\prime) \in \Gamma(u),\\
(0,0) \preceq (t^{\prime},j^{\prime}) \prec (t,j)\end{array}}} \!\!\!\!\!\!\!\!\!\!\!\!\!\!\!\!\gamma_2 (\abs{u(t^\prime,j^\prime)}) } \Bigg\} + \varepsilon . &
\end{align*}
This completes the poof. 
\end{proof}

\begin{remark}
The condition ($\mathrm{b}$) of Proposition~\ref{P:01} is not restrictive. With same augments as those in proof of \cite[Proposition 1]{Wang.2012}, one can provide sufficient conditions under which the condition ($\mathrm{b}$) of Proposition~\ref{P:01} holds. In particular, pre-iISS together with the Standing Assumptions are enough to get the desired property.
\hfill$\Box$
\end{remark}

Now we pass from semiglobal results to global results. The following theorem shows that semiglobal practical robust pre-iISS implies pre-iISS. 
\begin{proposition}\label{T:00}
Let $\A \subset {\X}$ be a compact set.
Assume that the hybrid system $\H$ is SPR-pre-iISS with respect to $\A$.
There exists an admissible perturbation radius $\sigma_2 \colon \X \to \Rp$ that is positive on $\X \backslash \A$ such that the hybrid system $\H_{\sigma_2}$, the $\sigma_2$-perturbation of $\H$, is pre-iISS with respect to $\A$.
\end{proposition}
\begin{proof} Inspired by the proof of \cite[Lemma 7.19]{Goebel.2012}, we show the conclusion.
According to the SPR-pre-iISS property of $\H$, let $\omega$ be a proper indicator for a compact set $\A$ on ${\X}$.
Also, let $\sigma_1 \colon {\X} \to \Rp$ be an admissible perturbation radius that is positive on ${\X} \backslash \A$.
Moreover, let the gain functions $\alpha \in \Kinf$, $\beta \in \KL$ and $\gamma_1,\gamma_2 \in \K$ come from Definition~\ref{D:02}.
Pick a sequence $\{ r_{m} \}_{m \in \Z}$ such that $r_{m+1} \geq 4 \beta(r_{m},0) \geq 4 r_{m} > 0$ for each $m \in \mathbb{Z}$, $\lim_{m \to - \infty} {r_{m} = 0}$ and $\lim_{m \to + \infty} {r_{m} = + \infty}$.
By SPR-pre-iISS with respect to $\mathcal{A}$, for each $m \in \mathbb{Z}$, there exists some $\delta_{m} \in (0,1)$ such that each solution pair $(\overline{x},u)$ to $\mathcal{H}_{\delta_{m} \sigma_{1}}$ with $u \in \mathcal{L}_{\gamma_{1},\gamma_{2}}^e (r_{m})$ and $\omega(\overline{\xi}) \leq r_{m}$ satisfies
\begin{align} \label{eq:r8}
\alpha(\omega(\overline{x}(t,j,\overline{\xi},u))) \leq & \beta (\omega(\overline{\xi}),t+j) + \int_0^t \gamma_1 (\abs{u(s,i(s))}) \mathrm{d}s \nonumber \\
& + \!\!\!\!\!\!\!\!\!\!\!\!\!\!\!\!{ \sum_{\scriptsize{\begin{array}{c}
(t^\prime,j^\prime) \in \Gamma(u),\\
(0,0) \preceq (t^{\prime},j^{\prime}) \prec (t,j)\end{array}}} \!\!\!\!\!\!\!\!\!\!\!\!\!\!\!\!\gamma_2 (\abs{u(t^\prime,j^\prime)} ) } + \frac{r_{m-1}}{2}
\end{align}
for all $(t,j) \in \dom \overline{x}$. The following also holds
\begin{align*} 
\alpha(\omega(\overline{x}(t,j,\overline{\xi},u))) \leq r_{m+1} + \int_0^t \gamma_1 (\abs{u(s,i(s))}) \mathrm{d}s + & \!\!\!\!\!\!\!\!\!\!\!\!\!\!\!\!{ \sum_{\scriptsize{\begin{array}{c}
(t^\prime,j^\prime) \in \Gamma(u),\\
(0,0) \preceq (t^{\prime},j^{\prime}) \prec (t,j)\end{array}}} \!\!\!\!\!\!\!\!\!\!\!\!\!\!\!\!\gamma_2 (\abs{u(t^\prime,j^\prime)}) } &
\end{align*}
for all $(t,j) \in \dom \overline{x}$. It follows with the estimate (\ref{eq:r8}) that there exists some $\tau_{m} > 0$ such that each solution pair $(\overline{x},u)$ with $u \in \mathcal{L}_{\gamma_{1},\gamma_{2}}^e (r_{m})$ and $\omega(\overline{\xi}) \leq r_{m}$ satisfies
\begin{align*}
\alpha(\omega(\overline{x}(t,j,\overline{\xi},u))) \leq & r_{m-1} + \int_0^t \gamma_1 (\abs{u(s,i(s))}) \mathrm{d}s + \!\!\!\!\!\!\!\!\!\!\!\!\!\!\!\!{\color{black} \sum_{\scriptsize{\begin{array}{c}
(t^\prime,j^\prime) \in \Gamma(u),\\
(0,0) \preceq (t^{\prime},j^{\prime}) \prec (t,j)\end{array}}} \!\!\!\!\!\!\!\!\!\!\!\!\!\!\!\!\gamma_2 (\abs{u(t^\prime,j^\prime)}) } & 
\end{align*}
for all $(t,j) \in \dom \overline{x}$ with $t+j \geq \tau_m$. Pick any admissible perturbation radius $\sigma_2 \colon \X \to \Rp$ that is positive on $\X \backslash \A$ such that $\sigma_2 (\overline{\xi}) \leq \min \{ \delta_{m-1},\delta_m , \delta_{m+1} \} \sigma_1 (\overline{\xi})$ for all $r_{m-1} \leq \omega(\overline{\xi}) \leq r_m$. Then, for every $m \in \mathbb{Z}$ and for each solution pair $(\overline{x},u)$ to $\H_{\sigma_2}$ with $u \in \mathcal{L}_{\gamma_1,\gamma_2}^e (r_{m})$ and $\omega(\overline{\xi}) \leq r_m$, the following hold
\begin{description}
    \item[(\emph{i})] $\alpha(\omega(\overline{x}(t,j,\overline{\xi},u))) \leq r_{m+1} + \int_0^t \gamma_1 (\abs{u(s,i(s))}) \mathrm{d}s + {\color{black} \sum_{\scriptsize{\begin{array}{c}
(t^\prime,j^\prime) \in \Gamma(u),\\
(0,0) \preceq (t^{\prime},j^{\prime}) \prec (t,j)\end{array}}}\!\!\!\!\!\!\!\gamma_2 (\abs{u(t^\prime,j^\prime)}) }$ for all $(t,j) \in \mathrm{dom}\overline{x}$.
\item[(\emph{ii})] There exists some $\tau_{m} > 0$ such that
\begin{align*}
\alpha(\omega(\overline{x}(t,j,\overline{\xi},u))) \leq & r_{m-1} + \int_0^t \gamma_1 (\abs{u(s,i(s))}) \mathrm{d}s + \!\!\!\!\!\!\!\!\!\!\!\!\!\!\!\!{\color{black} \sum_{\scriptsize{\begin{array}{c}
(t^\prime,j^\prime) \in \Gamma(u),\\
(0,0) \preceq (t^{\prime},j^{\prime}) \prec (t,j)\end{array}}} \!\!\!\!\!\!\!\!\!\!\!\!\!\!\!\!\gamma_2 (\abs{u(t^\prime,j^\prime)}) } \nonumber \\
& \qquad\qquad\qquad\quad \forall (t,j) \in \dom \overline{x} \textrm{ with } t+j \geq \tau_m .
\end{align*}
\end{description}
Let $\tilde{\beta} \colon \Rp \times \Rp \to \Rp$ be
\begin{align*} 
\tilde\beta (r,s) = \sup \Bigg\{ \alpha(\omega(\overline{x}(t,j,\overline{\xi},u))) & - \int_0^t \gamma_1 (\abs{u(s,i(s))}) \mathrm{d}s - \!\!\!\!\!\!\!\!\!\!\!\!\!\!\!\!{\color{black} \sum_{\scriptsize{\begin{array}{c}
(t^\prime,j^\prime) \in \Gamma(u),\\
(0,0) \preceq (t^{\prime},j^{\prime}) \prec (t,j)\end{array}}} \!\!\!\!\!\!\!\!\!\!\!\!\!\!\!\!\gamma_2 (\abs{u(t^\prime,j^\prime)}) } \colon \nonumber \\
& \overline{x} \in \varrho_{\mathcal{H}_{\sigma_{2}}} (\overline{\xi}) , \, u \in \mathcal{L}_{\gamma_1,\gamma_2}^e (r) , \, \omega(\overline{\xi}) \leq r , \, t+j \geq s \Bigg\} .
\end{align*}
By the very definition, $r \to \tilde{\beta}(r,s)$ is nondecreasing for each $s \geq 0$. We also get $\alpha(r) \leq \tilde{\beta} (r,0)$ for all $r \geq 0$. The item (\textit{i}) implies that $\tilde{\beta}(r,0)$ is bounded. By the definition of $\tilde{\beta}$, $s \to \tilde{\beta} (r,s)$ is nonincreasing, and so $\tilde{\beta} (r,s)$ is bounded for all $s \geq 0$. From the item (\textit{ii}), for each $r \geq 0$, we get $\tilde{\beta} (r,s) \to 0$ as $s \to +\infty$. So $\tilde{\beta}$ has all properties required of a $\mathcal{KL}$ function. Consequently, the existence of such a $\tilde\beta$ implies that for all $(t,j) \in \mathrm{dom}\,\overline{x}$, each solution pair $(\overline{x},u)$ to $\H_{\sigma_2}$ with $u \in \mathcal{L}_{\gamma_1,\gamma_2}^e$ and $\overline{\xi} \in \X$ satisfies
\begin{align*}
\alpha( \omega(\overline{x} (t,j,\overline{\xi},u))) & \leq \tilde{\beta} (\omega(\overline{\xi}),t+j) + \int_0^t \gamma_1 (\abs{u(s,i(s))}) \mathrm{d}s + \!\!\!\!\!\!\!\!\!\!\!\!\!\!\!\!{\color{black} \sum_{\scriptsize{\begin{array}{c}
(t^\prime,j^\prime) \in \Gamma(u),\\
(0,0) \preceq (t^{\prime},j^{\prime}) \prec (t,j)\end{array}}} \!\!\!\!\!\!\!\!\!\!\!\!\!\!\!\!\gamma_2 (\abs{u(t^\prime,j^\prime)}) } \nonumber \\
& \leq \max \Bigg\{ 3 \tilde{\beta} (\omega(\overline{\xi}),t+j) , 3 \int_0^t \gamma_1 (\abs{u(s,i(s))}) \mathrm{d}s , 3 \!\!\!\!\!\!\!\!\!\!\!\!\!\!\!\!{\color{black} \sum_{\scriptsize{\begin{array}{c}
(t^\prime,j^\prime) \in \Gamma(u),\\
(0,0) \preceq (t^{\prime},j^{\prime}) \prec (t,j)\end{array}}} \!\!\!\!\!\!\!\!\!\!\!\!\!\!\!\!\gamma_2 (\abs{u(t^\prime,j^\prime)}) } \Bigg\} .
\end{align*}
Setting $\overline{\beta} (\cdot,\cdot) := 3 \tilde{\beta} (\cdot,\cdot), \overline{\gamma}_1 (\cdot) := 3 \gamma_1 (\cdot)$ and $\overline{\gamma}_2 (\cdot) := 3 \gamma_2 (\cdot)$ completes the proof. 
\end{proof}
As seen, the combination of Proposition~\ref{P:01} and Proposition~\ref{T:00} shows robustness of pre-iISS in terms of sufficiently small perturbations. This finishes the proof.
\section{Proof of Lemma~\ref{L:15}} \label{Ap:K}
From the definition of $V_0$, it is easy to see that $V_0 (x) |_{x \in \A} = 0$.
The first inequality of \eqref{eq:e56} comes from considering the special case of \eqref{eq:e53} in which $(t,j) = (0,0)$.
The second one follows from the first inequality of \eqref{eq:e10} with $\overline{x}_\varphi$ and $\overline{\beta}_0$ in place of $x_\varphi$ and $\tilde{\beta}_0$, and the fact that $\abs{d(t,j)} \leq 1$ for all $(t,j) \in \dom \, d$. Let 
\begin{align} \label{eq:e84}
\tilde{d} (t,j) := \left\{ \begin{array} {l}
 \mu \quad \quad \quad \quad \quad \quad \;\,\;\, \mathrm{for} \; (t,j) \preceq (h,m) , \\ 
 d (t-h , j - m) \quad \; \mathrm{for} \; (t,j) \succ (h,m) . \\ 
 \end{array} \right.
\end{align}
Pick a maximal solution $\overline{x}_{\varphi} (h,m,\xi)$ to $\hat{\mathcal{H}}_\sigma$.
Note that if $\overline{x}_{\varphi} (t,j,\overline{x}_{\varphi} (h,m,\xi)) \in \hat{\varrho}_{\sigma} (\overline{x}_{\varphi} (h,m,\xi))$ then there exists $\overline{x}_{\varphi} (t+h,j+m,\xi) \in \hat{\varrho}_{\sigma} (\xi)$ such that $(t,j) \in \mathrm{dom}\,\overline{x}_{\varphi}$ implies $(t+h,j+m) \in \mathrm{dom}\,\overline{x}_{\varphi}$ and $\overline{x}_{\varphi} (t,j,\overline{x}_{\varphi} (h,m,\xi)) = \overline{x}_{\varphi} (t + h ,j + m,\xi)$.
Given $m = 0$, one can see that for each $\xi \in \hat{\mathcal{C}} \backslash \mathcal{A}, \, \abs{\mu} \leq 1 $, $\overline{x}_{\varphi} \in \hat{\varrho}_{\sigma} (\xi)$ and $(h,0),(t,0) \in \mathrm{dom} \, \overline{x}_{\varphi}$ such that $(h,0) \prec (t,0)$, we have
\begin{align*}
V_{0} (\overline{x}_\varphi (h,0,\xi)) = & \sup_{\overline{x}_\varphi \in \hat{\varrho}_\sigma (\overline{x}_\varphi (h,0,\xi)) , \, (t,j) \in \dom \, \overline{x}_\varphi , \, d \in \overline{\mathcal{M}}} \Big\{ \alpha( \omega (\overline{x}_\varphi (t,j,\overline{x}_{\varphi} (h,0,\xi)) ) ) & \nonumber \\
& - \int_0^t \overline{\gamma}_1 ( \abs{d(s,i(s))} \, \varphi(\omega(\overline{x}_\varphi (s,i(s),\overline{x}_\varphi (h,0,\xi)) ) ) ) \mathrm{d} s & \nonumber \\
& - \!\!\!\!\!\!\!\!\!\!\!\!\!\!\!\sum_{\scriptsize{\begin{array}{c}(t^\prime,j^\prime) \in \Gamma(d),\\(0,0) \preceq (t^{\prime},j^{\prime}) \prec (t,j)\end{array}}} \!\!\!\!\!\!\!\!\!\!\!\!\!\!\! \overline\gamma_2 (\abs{d(t^\prime,j^\prime)} \varphi(\omega (\overline{x}_{\varphi}(t^\prime,j^\prime,\overline{x}_{\varphi} (h,0,\xi))) ))\Big\}
\\
= & \sup_{\overline{x}_{\varphi} \in \hat{\varrho}_{\sigma} (\xi) , \, (t+h,j) \in \dom \, \overline{x}_{\varphi} , \, d \in \overline{\mathcal{M}}} \Big\{ \alpha( \omega ( \overline{x}_{\varphi} (t+h,j,\xi) ) ) & \nonumber \\
& - \int_0^t \overline{\gamma}_{1}(\abs{d(s,i(s))} \, \varphi (\omega(\overline{x}_{\varphi} (s+h,i(s),\xi)))) \mathrm{d} s & \nonumber \\
& - \!\!\!\!\!\!\!\!\!\!\!\!\!\!\! \sum_{\scriptsize{\begin{array}{c}(t^\prime,j^\prime) \in \Gamma(d),\\(0,0) \preceq (t^{\prime},j^{\prime}) \prec (t,j)\end{array}}} \!\!\!\!\!\!\!\!\!\!\!\!\!\!\! \overline\gamma_2 (\abs{d(t^\prime,j^\prime)} \varphi(\omega (\overline{x}_{\varphi}(t^\prime + h,j^\prime,\xi)) ))\Big\} \\
= & \sup_{\overline{x}_{\varphi} \in \hat{\varrho}_{\sigma} (\xi) , \, (\tau,j) \in \mathrm{dom} \, \overline{x}_{\varphi} , \tau \geq h , d \in \overline{\mathcal{M}}} \Big\{ \alpha( \omega ( \overline{x}_{\varphi} (\tau,j,\xi) ) ) & \nonumber \\
& - \int_{h}^{\tau_{1}} \overline{\gamma}_{1}(\abs{d(s_{1} - h,0)} \, \varphi ( \omega ( \overline{x}_{\varphi} (s_{1},0,\xi) ) ) ) \mathrm{d} s_{1} & \nonumber \\
& - \int_{\tau_1}^\tau \overline{\gamma}_{1} (\abs{d(s_{1} - h,i(s_1))} \, \varphi ( \omega ( \overline{x}_{\varphi} (s_{1},i(s_1),\xi) ) ) ) \mathrm{d} s_1 & \nonumber \\
& - \!\!\!\!\!\!\!\!\!\!\!\!\!\!\! \sum_{\scriptsize{\begin{array}{c}(\tau^\prime-h,j^\prime) \in \Gamma(d),\\(0,0) \preceq (\tau^{\prime}-h,j^{\prime}) \prec (\tau-h,j)\end{array}}} \!\!\!\!\!\!\!\!\!\!\!\!\!\!\! \overline\gamma_2 (\abs{d(\tau^\prime-h,j^\prime)} \varphi(\omega (\overline{x}_{\varphi}(\tau^\prime,j^\prime,\xi)) ))\Big\} \\
\leq & \sup_{\overline{x}_\varphi \in \hat{\varrho} (\xi) , \, (\tau,j) \in \dom \, \overline{x}_\varphi , \tau \geq 0 , \tilde{d} \in \overline{\mathcal{M}}} \Big\{ \alpha( \omega ( \overline{x}_\varphi (\tau,j,\xi))) & \nonumber \\
& - \int_0^\tau \overline{\gamma}_1 (|\tilde{d} (s_1,i(s_1))| \, \varphi ( \omega ( \overline{x}_\varphi (s_1,i(s_1),\xi) ) ) ) \mathrm{d} s_1 & \nonumber \\
& - \!\!\!\!\!\!\!\!\!\!\!\!\!\!\! \sum_{\scriptsize{\begin{array}{c}(\tau^\prime,j^\prime) \in \Gamma(\tilde d),\\(0,0) \preceq (\tau^{\prime},j^{\prime}) \prec (\tau,j)\end{array}}} \!\!\!\!\!\!\!\!\!\!\!\!\!\!\! \overline\gamma_2 (|\tilde{d}(\tau^\prime,j^\prime)| \varphi(\omega (\overline{x}_{\varphi}(\tau^\prime,j^\prime,\xi)) ))\Big\} \\
& + \int_{0}^{h} \overline{\gamma}_{1}(\abs{\mu} \, \varphi ( \omega ( \overline{x}_\varphi (s_1,0,\xi) ) ) ) \mathrm{d} s_1 & \nonumber\\
= & V_{0} (\xi) + \int_{0}^{h} \overline{\gamma}_{1} (\abs{\mu} \, \varphi ( \omega ( \overline{x}_{\varphi} (s_{1},0,\xi) ) ) ) \mathrm{d} s_1 . & 
\end{align*}
For each $\xi \in \hat{\mathcal{D}}$ and $g \in \hat{G}(\xi)$, there exists $\overline{x}_{\varphi} \in \hat{\varrho}_{\sigma} (\xi)$ such that $(0,1) \in \mathrm{dom} \, \overline{x}_{\varphi}$, that is, $t_{0} = t_{1}$. So let $(h,m) = (0,1)$. We have for any any $\xi \in \hat{\mathcal{D}}$ and $g \in \hat{G}(\xi)$, and $\abs{\mu} \leq 1$
\begin{align*}
V_{0} (\overline{x}_{\varphi} (0,1,\xi)) = & \sup_{\overline{x}_{\varphi} \in \hat{\varrho}_{\sigma} (\xi) , \, (t,j) \in \mathrm{dom} \, \overline{x}_{\varphi} , \, d \in \overline{\mathcal{M}}} \Big\{ \alpha( \omega (\overline{x}_{\varphi} (t,j,\overline{x}_{\varphi} (0,1,\xi)))) & \nonumber \\
& - \int_0^t \overline{\gamma}_{1}( \abs{d(s,i(s))} \, \varphi(\omega(\overline{x}_\varphi (s,i(s),\overline{x}_{\varphi} (0,1,\xi)) ) ) ) \mathrm{d} s & \nonumber \\
& - \!\!\!\!\!\!\!\!\!\!\!\!\!\!\!\sum_{\scriptsize{\begin{array}{c}(t^\prime,j^\prime) \in \Gamma(d),\\(0,0) \preceq (t^{\prime},j^{\prime}) \prec (t,j)\end{array}}} \!\!\!\!\!\!\!\!\!\!\!\!\!\!\! \overline\gamma_2 (\abs{d(t^\prime,j^\prime)} \varphi(\omega (\overline{x}_{\varphi}(t^\prime,j^\prime,\overline{x}_{\varphi} (0,1,\xi))) ))\Big\}\\
= & \sup_{\overline{x}_{\varphi} \in \hat{\varrho}_{\sigma} (\xi) , \, (t,j+1) \in \mathrm{dom} \, \overline{x}_{\varphi} , \, d \in \overline{\mathcal{M}}} \Big\{ \alpha( \omega ( \overline{x}_{\varphi} (t,j+1,\xi) ) ) & \nonumber \\
& - \int_0^t \overline{\gamma}_1 (\abs{d(s,i(s))} \, \varphi (\omega(\overline{x}_\varphi (s,i(s)+1,\xi)))) \mathrm{d} s & \nonumber \\
& - \!\!\!\!\!\!\!\!\!\!\!\!\!\!\!\sum_{\scriptsize{\begin{array}{c}(t^\prime,j^\prime) \in \Gamma(d),\\(0,0) \preceq (t^{\prime},j^{\prime}) \prec (t,j)\end{array}}} \!\!\!\!\!\!\!\!\!\!\!\!\!\!\! \overline\gamma_2 (\abs{d(t^\prime,j^\prime)} \varphi(\omega (\overline{x}_{\varphi}(t^\prime,j^\prime+1,\xi)) ))\Big\}\\
= & \sup_{\overline{x}_{\varphi} \in \hat{\varrho}_{\sigma} (\xi) , \, (t,\ell) \in \mathrm{dom} \, \overline{x}_\varphi , \ell \geq 1 , d \in \overline{\mathcal{M}}} \Big\{ \alpha(\omega(\overline{x}_\varphi (t,\ell,\xi))) & \nonumber \\
& - \int_{t_1}^t \overline{\gamma}_1 (\abs{d(s,i(s)-1)} \, \varphi (\omega(\overline{x}_\varphi (s,i(s),\xi)))) \mathrm{d} s & \nonumber \\
& - \!\!\!\!\!\!\!\!\!\!\!\!\!\!\!\!\!\!\!\sum_{\scriptsize{\begin{array}{c}(t^\prime,\ell^\prime-1) \in \Gamma(d),\\(0,0) \preceq (t^{\prime},\ell^{\prime}-1) \prec (t,\ell-1)\end{array}}} \!\!\!\!\!\!\!\!\!\!\!\!\!\!\!\!\!\!\!\!\! \overline\gamma_2 (\abs{d(t^\prime,\ell^\prime-1)} \varphi(\omega (\overline{x}_{\varphi}(t^\prime,\ell^\prime,\xi)) ))\Big\}\\
\leq & \sup_{\overline{x}_{\varphi} \in \hat{\varrho}_\sigma (\xi) , \, (t,\ell) \in \mathrm{dom} \overline{x}_{\varphi} , \ell \geq 0, \tilde{d} \in \overline{\mathcal{M}}} \Big\{ \alpha( \omega ( \overline{x}_{\varphi} (t,\ell,\xi) ) ) & \nonumber \\
& - \int_0^t \overline{\gamma}_1 (|\tilde{d} (s,i(s))| \, \varphi (\omega(\overline{x}_\varphi (s,i(s),\xi)))) \mathrm{d} s & \nonumber \\
& - \!\!\!\!\!\!\!\!\!\!\!\!\!\!\sum_{\scriptsize{\begin{array}{c}(t^\prime,\ell^\prime) \in \Gamma(\tilde{d}),\\(0,0) \preceq (t^{\prime},\ell^{\prime}) \prec (t,\ell)\end{array}}} \!\!\!\!\!\!\!\!\!\!\!\!\!\!\!\! \overline\gamma_2 (|\tilde{d}(t^\prime,\ell^\prime)| \varphi(\omega (\overline{x}_{\varphi}(t^\prime,\ell^\prime,\xi)) ))\Big\}\\
& + \overline{\gamma}_{2}(\abs{\mu} \, \varphi ( \omega ( \xi ) ) )  & \nonumber\\
= & V_{0} (\xi) + \overline{\gamma}_{2}(\abs{\mu} \, \varphi ( \omega  (\xi) ) ) & \nonumber
\end{align*}
Therefore, for each $\xi \in \hat{\D}$ and $g \in \hat{G}(\xi)$, and $\abs{\mu} \leq 1$, we get
\begin{align*}
V_0 (g) \leq V_0 (\xi) + \overline{\gamma}_2 (\abs{\mu} \, \varphi ( \omega ( \xi ) ) ) . &
\end{align*}
This completes the proof.
\section{Proof of Lemma~\ref{L:5}} \label{Ap:G}
Sufficiency immediately follows from (\ref{eq:e20}) and (\ref{eq:e33}) with $\nu = 0$.
To establish necessity, by the Converse Lyapunov Theorem \cite[Theorem 3.13]{Cai.2008}, there exist a smooth Lyapunov function $V$ and $\alpha_1,\alpha_2, \alpha_3 \in \Kinf$ such that 
\begin{equation*}
\begin{array}{cccl}
 \alpha_1 (\omega (\xi)) \leq V(\xi) &\leq& \alpha _2 (\omega(\xi)) &\qquad \forall \xi \in \X , \\
 \left\langle \nabla V(\xi),f(\xi,0) \right\rangle &\leq& - \alpha_3 \left( \omega (\xi) \right) & \qquad\forall \xi \in \C , \\
V(g(\xi,0)) - V(\xi) &\leq& - \alpha_3 (\omega(\xi)) & \qquad\forall \xi \in \D .  
\end{array}
\end{equation*}
Define the following continuous function $\delta \colon \Rp \times \Rp \to \R$ by
\begin{align*}
\delta(s,r) := \max \Big\{ & \max \big\{ \left\langle \nabla V(\xi),f(\xi,\mu) \right\rangle + \alpha_{3} (\abs{\xi}) \colon \omega(\xi) = s , \abs{\mu} = r \big\} , & \nonumber \\
& \max \big\{ V(g(\xi,\mu)) - V(\xi) + \alpha_{3} (\abs{\xi}) \colon \omega(\xi) = s , \abs{\mu} = r \big\} \Big\} . &
\end{align*}
It should be pointed out that $\delta(s,0) < 0$ for all $s>0$ as $\mathcal{H}$ is 0-input pre-AS. Applying Lemma 3.1 in \cite{Sontag.1990} to $\delta (\cdot,\cdot)$ gives that there exist some $\chi \in \Kinf$ and a smooth function $q \colon \Rp \to \Rsp$ such that
\begin{itemize}
    \item[($\textit{a}$)] $q(s) \neq 0$ for all $s \geq 0$ and $q(s) \equiv 1$ for all $s \in [0,1]$;
\item[($\textit{b}$)] $\delta (s,p) < 0$ for each pair $(s,r) \in \mathbb{R}_{\geq 0} \times \Rp$ for which $\chi(r) < s$, and each $p \leq q(s) r$.
\end{itemize}
We use these properties to establish that (\ref{eq:e20}) and (\ref{eq:e33}) hold. Let $I$ be the $m \times m$ identity matrix.
Assume that $\omega(\xi) > \chi (\abs{\nu})$ and let $s := \omega(\xi)$ and $r := \abs{\nu}$. By the very definition of $\delta$,
\begin{align*}
\max \{ \left\langle \nabla V(\xi),f(\xi,q(\omega(\xi)) I \nu) \right\rangle + \alpha_3 (\abs{\xi}) , V(g(\xi,q(\omega(\xi)) I \nu)) - V(\xi) + \alpha_{3} (\abs{\xi}) \} \leq \delta (s,p)
\end{align*}
where
\begin{align*}
p = \abs{q(\omega(\xi)) I \nu} \leq q(s) r .
\end{align*}
It follows from the fact that $s > \chi(r)$ and using the item ($\textit{b}$) that $\delta$ is negative everywhere. This completes the proof.


\end{document}